\documentclass{article}

\usepackage{amsmath}
\usepackage{amssymb}
\usepackage{amsfonts}
\usepackage{amsthm}
\usepackage{cleveref}
\usepackage{fullpage}
\usepackage{xcolor}
\usepackage{graphicx}

\newcommand{\R}{\mathbb{R}}
\newcommand{\N}{\mathbb{N}}
\renewcommand{\d}{\mathrm{d}}
\renewcommand{\div}{\mathrm{div}}
\newcommand{\hd}{\mathcal{H}}
\newcommand{\restr}{{\mbox{\LARGE$\llcorner$}}}
\newcommand{\meas}{{\mathcal M}}
\newcommand{\prob}{{\mathcal P}}
\newcommand{\argmin}{{\mathrm{argmin}}}
\newcommand{\BV}{\mathrm{BV}}

\newtheorem{theorem}{Theorem}
\newtheorem{lemma}[theorem]{Lemma}

\newtheorem{definition}[theorem]{Definition}
\newtheorem{example}[theorem]{Example}
\newtheorem{remark}[theorem]{Remark}
\newtheorem{problem}[theorem]{Problem}

\graphicspath{{./figures/},{.}}

\title{Phase field models for two-dimensional branched transportation problems}
\author{Benedikt Wirth}
\date{}

\begin{document}

\maketitle

\begin{abstract}
We analyse the following inverse problem.
Given a nonconvex functional (from a specific, but quite general class) of normal, codimension-1 currents (which in two spatial dimensions can be interpreted as transportation networks),
find the potential of a phase field energy which approximates the given functional.
We prove existence of a solution as well as its characterization via a linear deconvolution problem.
We also provide an explicit formula that allows to approximate the solution arbitrarily well in the supremum norm.
\end{abstract}

\section{Introduction}
% two-sentence summary and contributions (existence, reduce problem to linear deconvolution problem, explicitly solvable for cases of interest and numerically solvable e.g. if $g$ is zero for large enough arguments, i.e. $\tau$ starts linear; explicit solution for piecewise linear $\tau$, which can be used as approximation)
%
% mention that the phase field energy also can be used for concave functions of the jump part of the gradient of an SBV function, thus yielding MS energies

We consider the following inverse problem.

\begin{problem}[Inverse problem for phase field cost]\label{pbm:inverseProblem}
Given a continuous, nondecreasing, concave function $\tau:[0,\infty)\to[0,\infty)$ with $\tau(0)=0$,
find a function $c:[0,\infty)\to[0,\infty)$ such that for all $w\geq0$ we have
\begin{equation}\label{eqn:inverseProblem}
\tau(w)=\inf\left\{F^c[\psi]\,\middle|\,\int_\R\psi\,\d y=w\right\}
\quad\text{for }
F^c[\psi]=\int_\R\frac12|\psi'(y)|^2+c\left(|\psi(y)|\right)\,\d y\,.
\end{equation}
\end{problem}

As will be explained further below in the introduction, the sought $c$ plays the same role as the double-well potential in the well-known Modica--Mortola phase field functional.
In more detail, we seek the potential $c$ of a phase field functional which shall approximate a given cost function of two-dimensional transportation networks
(in which $\tau(w)$ represents the cost per unit length of a network edge carrying a flux $w$).
The interest in such so-called branched transport functionals grew tremendously during the past years (an introduction can be found in \cite[\S4.4.2]{Sa15}),
and phase field approximations represent a promising method to numerically compute optimal transport networks.
Two particular instances of branched transport functionals have been approximated via phase fields in \cite{OuSa11,ChaFerMer16},
while in the current article we present phase field functionals for a much larger class of branched transportation functionals.
In fact, any nonnegative lower semi-continuous integral functional of normal codimension-1 currents $\sigma$ can be written as $\int\tau(\sigma)$
for a subadditive $\tau$ (and the functions $\tau$ considered in \cref{pbm:inverseProblem} form a large important subclass),
so the results of this article are not only relevant for branched transportation in two dimensions,
but also for the approximation of nonconvex integral functionals of BV-functions
(such as variants of the well-known Mumford--Shah energy, used in image processing or fracture mechanics, in which the jump part of the BV-function gradient is penalized by some $\tau$).
Nevertheless, in our exposition we chose to motivate \cref{pbm:inverseProblem} via branched transportation models.

Throughout this article we will call $\tau:[0,\infty)\to[0,\infty)$ a \emph{transport cost} and $c:[0,\infty)\to[0,\infty)$ a \emph{phase field cost};
we will furthermore make use of the \emph{mass-specific phase field cost} $z:\phi\mapsto c(\phi)/\phi$.
We will call a transport cost $\tau$ \emph{admissible} if it satisfies the properties listed in \cref{pbm:inverseProblem}.

\begin{definition}[Admissible tansport cost]
A transport cost $\tau:[0,\infty)\to[0,\infty)$ is \emph{admissible} if it is nondecreasing, concave, and continuous with $\tau(0)=0$.
\end{definition}

The article presents a comprehensive analysis of inverse \cref{pbm:inverseProblem}, including the following main results.
\begin{itemize}
\item
\Cref{pbm:inverseProblem} has a (not necessarily unique) solution $c$ (\cref{thm:existence}).
We furthermore show structural properties of particular solutions,
most importantly that the corresponding mass-specific phase field cost $z$ is lower semi-continuous and nonincreasing.
\item
Any transport cost $\tau$ that is induced via \eqref{eqn:inverseProblem} by a Borel measurable phase field cost $c:[0,\infty)\to[0,\infty)$
is admissible (\cref{thm:tauProperties}).
Thus, the restriction of \cref{pbm:inverseProblem} to admissible transport costs $\tau$ is natural
(note that all transport costs $\tau$ of interest are nondecreasing subadditive functions with $\tau(0)=0$,
of which admissible ones form a large important subclass).
\item
For the examples of piecewise affine $\tau$ and of degree $\alpha$-homogeneous $\tau$, $\alpha\in(0,1)$,
explicit formulae for a solution $c$ to \cref{pbm:inverseProblem} are given (\cref{exm:branchedTransport,thm:genUrbPln}).
Obviously the formula for piecewise affine $\tau$ lends itself for approximating arbitrary concave transport costs $\tau$ (for instance via linear interpolation).
\item
\Cref{pbm:inverseProblem} can be transformed into an (almost) equivalent linear deconvolution problem in the following sense.
Let $[-\tau(-\cdot)]^\ast$ denote the Legendre--Fenchel conjugate of $-\tau(-\cdot)$ (where we extended $\tau$ to $(-\infty,0)$ by $-\infty$) and introduce the function
\begin{equation}\label{eqn:convFunction}
r(s)=\begin{cases}\frac{2\sqrt2}{3\sqrt{-s}}&\text{if }s<0,\\0&\text{else.}\end{cases}
\end{equation}
Furthermore, introduce the nonlinear transformation
\begin{equation}\label{eqn:nonlinearTransformation}
g=(z^{-1}(\cdot))^{3/2}
\end{equation}
of the mass-specific phase field cost $z$
(from which $z$ can be recovered as $z=g^{-1}((\cdot)^{2/3})$ and in which the inversion is meant in a generalized sense, see \cref{thm:generalizedInverse}).
Then, if the phase field cost $c$ solves \cref{pbm:inverseProblem}, we have
\begin{equation}\label{eqn:convolutionProblem}
[-\tau(-\cdot)]^\ast(t)=[g*r](t)
\end{equation}
for all $t$ from a certain subset of $[0,\infty)$ (necessary condition, \cref{thm:necessaryCondition}).
Conversely, if \eqref{eqn:convolutionProblem} holds for all $t\geq0$, then $c$ solves \cref{pbm:inverseProblem} (sufficient condition, \cref{thm:sufficientCondition}).
\end{itemize}

The remainder of the introduction describes in more detail the motivation of \cref{pbm:inverseProblem} via branched transport phase field functionals
and briefly derives the equivalent linear deconvolution problem \eqref{eqn:convolutionProblem} via a formal argument.
\Cref{sec:examples} provides examples of transport costs $\tau$ and corresponding phase field costs $c$,
some of which are derived as applications of the linear deconvolution problem \eqref{eqn:convolutionProblem}.
Existence of a solution to \cref{pbm:inverseProblem} and its properties are derived in \cref{sec:existence},
while \cref{sec:obtainableCosts} characterizes the class of phase field costs $\tau$ obtainable by \eqref{eqn:inverseProblem}.
Finally, \cref{sec:existencePhaseField} discusses properties of the minimizers $\psi$ inside \eqref{eqn:inverseProblem},
which is used in \cref{sec:deconvolution} to rigorously derive the linear deconvolution problem \eqref{eqn:convolutionProblem}.

\paragraph{Branched transportation.}
Classical optimal transport is concerned with finding the most cost-efficient way of transporting mass from a given initial mass distribution $\mu_0$ (represented as a probability measure in $\R^n$) to a given final mass distribution $\mu_1$.
Branched transportation is a variant of optimal transport in which the corresonding transportation cost favours transport in bulk \cite[\S4.4.2]{Sa15}.
This automatically leads to the creation of hierarchical transportation networks
in which mass from $\mu_0$ is gradually collected on the finer network branches, is then transported efficiently in bulk along big branches, and is finally distributed towards $\mu_1$ again on finer network branches (see \cref{fig:branchedTransport}).

Mathematically the problem is formulated as follows (see for instance \cite[Prop.\,2.32]{BrWi17}).
Let $\Omega\subset\R^d$ be a closed Lipschitz domain, and denote by $\meas(\Omega;\R^d)$ the set of $\R^d$-valued Radon measures on $\Omega$ and by $\prob(\Omega)$ the set of probability measures on $\Omega$.
Given $\mu_0,\mu_1\in\prob(\Omega)$, the mass flux from $\mu_0$ to $\mu_1$ is described by some $\sigma\in\meas(\Omega;\R^d)$ satisfying
\begin{equation}\label{eqn:divergenceConstraint}
\div\sigma=\mu_0-\mu_1
\end{equation}
in the distributional sense.
It is known that such a mass flux can be decomposed into a rectifiable and a diffuse part,
\begin{equation*}
\sigma=w\theta\hd^1\restr S+\sigma^\perp\,,
\end{equation*}
where $S\subset\Omega$ is a countably $1$-rectifiable set, $\hd^1$ denotes the one-dimensional Hausdorff measure, $w:S\to[0,\infty)$ is the locally transported mass, $\theta:S\to\R^d$ is the approximate tangent to $S$,
and $\sigma^\perp$ consists of a Lebesgue-continuous and a Cantor part.
The cost associated with mass flux $\sigma$ is then given by
\begin{equation*}
E^\tau[\sigma]=\int_\Omega\tau(\sigma)=\int_S\tau(w)\,\d\hd^1+\tau'(0)|\sigma^\perp|(\Omega)\,,
\end{equation*}
where $\tau:[0,\infty)\to[0,\infty)$ is nondecreasing and concave with $\tau(0)=0$ ($\tau'(0)\in[0,\infty]$ denotes the right derivative of $\tau$ in $0$, and $|\cdot|$ stands for the total variation measure).
The function $\tau(w)$ represents the cost for transporting mass $w$ by one unit distance.
The cost $E^\tau$ is now minimized over all mass fluxes from $\mu_0$ to $\mu_1$ to obtain the optimal mass flux and the minimal branched transport cost.

\begin{figure}
\setlength\unitlength\linewidth
\includegraphics[width=\unitlength]{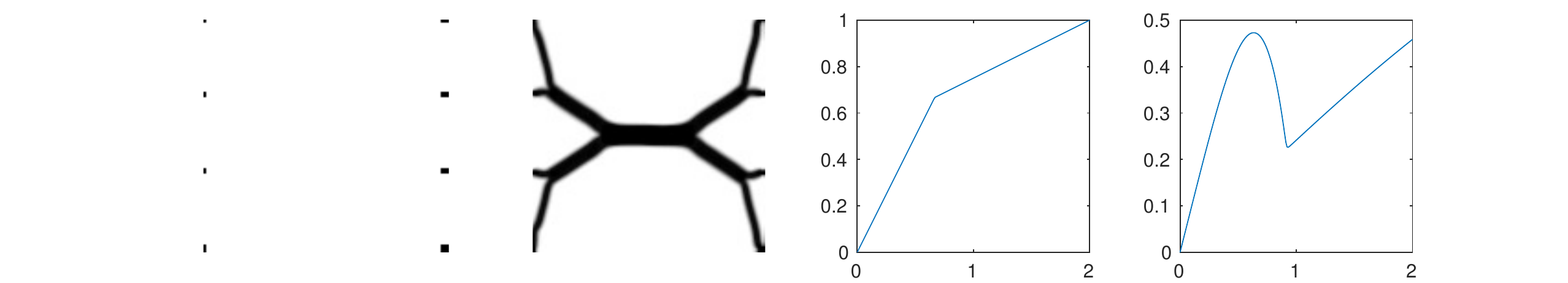}%
\begin{picture}(0,0)(1,0)
\put(.135,.015){$+$}
\put(.135,.065){$+$}
\put(.135,.113){$+$}
\put(.135,.16){$+$}
\put(.26,.015){$-$}
\put(.26,.065){$-$}
\put(.26,.113){$-$}
\put(.26,.16){$-$}
\put(.61,-.015){$w$}
\put(.5,.066){\rotatebox{90}{$\tau(w)$}}
\put(.82,-.015){$\phi$}
\put(.71,.066){\rotatebox{90}{$c(\phi)$}}
\end{picture}%
\caption{Optimal transportation network, numerically computed via the phase field approximation suggested in this article.
>From left to right: initial and final mass distribution $\mu_0$ (+) and $\mu_1$ (--), phase field approximation $\tilde\sigma$ of the optimal flux $\sigma$ (computed by minimizing \eqref{eqn:phaseFieldEnergy} subject to \eqref{eqn:divergenceConstraint} using a finite difference discretization in which $\tilde\sigma$ is represented as $\nabla u^\perp$ for some scalar function $u$; only the flux magnitude $|\tilde\sigma|$ is shown),
the employed transport cost $\tau$, and the corresponding phase field cost $c$ (which will be derived in \cref{exm:urbanPlanning}).
Note that rather than being concentrated along one-dimensional lines, the flux is a little diffused over a width $\varepsilon$ due to the phase field approximation.}
\label{fig:branchedTransport}
\end{figure}
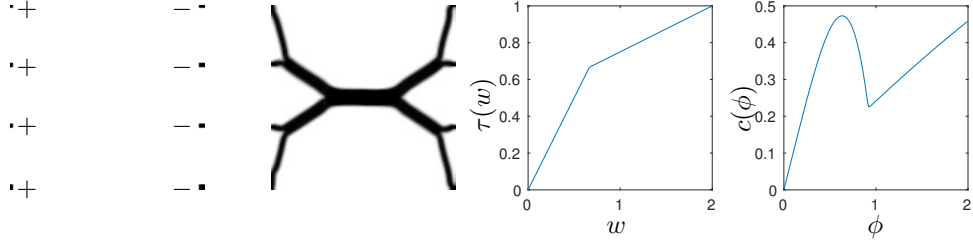

\paragraph{Phase field approximation.}
One approach to numerically find optimal transportation networks (that is, minimizers of $E^\tau$ subject to \eqref{eqn:divergenceConstraint})
consists in approximating $E^\tau$ by a smooth phase field functional $E^\tau_\varepsilon$ that is easier to minimize
(phase field models for special cases of branched transportation are proposed in \cite{OuSa11,FeRoWi18}).
In such an approach the flux $\sigma\in\meas(\Omega;\R^d)$ is replaced by a smoothed version $\tilde\sigma:\Omega\to\R^d$, called a phase field, where the degree of smoothing is determined by a small parameter $\varepsilon>0$ (see \cref{fig:branchedTransport}).
In the limit $\varepsilon\to0$ one aims to recover the original optimization problem in the sense that $E^\tau_\varepsilon$ $\Gamma$-converges to $E^\tau$.

In two spatial dimensions, a possible ansatz for the phase field functional $E^\tau_\varepsilon$, also followed in \cite{OuSa11}, is given by
\begin{equation}\label{eqn:phaseFieldEnergy}
E^\tau_\varepsilon[\tilde\sigma]
=\int_\Omega\frac{\varepsilon^\alpha}2|\nabla\tilde\sigma(x)|^2+\frac1{\varepsilon^\beta}c\left(\varepsilon^\gamma|\tilde\sigma(x)|\right)\,\d x
\end{equation}
for suitable exponents $\alpha,\beta,\gamma\in\R$ and a suitable phase field cost $c:[0,\infty)\to[0,\infty)$ with $c(0)=0$.
The aim of this article is to identify, for given $\tau$, the phase field cost $c$ such that the phase field functional indeed approximates $E^\tau$.
The essential idea behind such a phase field model is that along a single network branch the mass flux $\sigma$ and also its phase field approximation $\tilde\sigma$ are constant,
thereby reducing the problem dimension by one.
In more detail, consider a flux $\sigma$ moving mass $w\geq0$ upwards along the vertical axis,
\begin{equation*}
\sigma=w{0\choose1}\hd^1\restr\{0\}\times\R
\end{equation*}
(this flux should be thought of as a single branch of the mass flux obtained from minimizing $E^\tau$ subject to \eqref{eqn:divergenceConstraint};
indeed, if one zooms in far enough on such a branch it appears arbitrarily long, and without loss of generality we can choose coordinates such that the branch points upwards;
the same reasoning will be applied to all branches).
The corresponding phase field function $\tilde\sigma$ (the minimizer of $E^\tau_\varepsilon$) will then also be constant along the vertical direction and will be of the form
\begin{equation*}
\tilde\sigma(x_1,x_2)=\tilde w(x_1){0\choose1}
\end{equation*}
for some function $\tilde w:\R\to[0,\infty)$.
Since $\tilde\sigma$ is just a diffused version of $\sigma$, the magnitude of the vertical mass flux encoded by $\tilde\sigma$ must equal $w$,
and the phase field energy per unit interval (without loss of generality along $x_2\in[0,1]$) must equal the transportation cost $\tau(w)$,
\begin{align*}
w
&=\int_\R{0\choose1}\cdot\tilde\sigma\,\d x_1
=\int_\R\tilde w(x_1)\,\d x_1\,,\\
\tau(w)
&=\int_0^1\int_\R\frac{\varepsilon^\alpha}2|\nabla\tilde\sigma(x)|^2+\frac1{\varepsilon^\beta}c\left(\varepsilon^\gamma|\tilde\sigma(x)|\right)\,\d x_1\,\d x_2
=\int_\R\frac{\varepsilon^\alpha}2|\tilde w'(x_1)|^2+\frac1{\varepsilon^\beta}c\left(\varepsilon^\gamma|\tilde w(x_1)|\right)\,\d x_1\,.
\end{align*}
Introducing rescaled variables according to $y=x_1/\varepsilon^\beta$ and $\psi(y)=\tilde w(y\varepsilon^\beta)\varepsilon^\beta$ we obtain
\begin{align*}
w
&=\int_\R\psi(y)\,\d y\,,\\
\tau(w)
&=\int_\R\frac{\varepsilon^{\alpha-3\beta}}2|\psi'(y)|^2+c\left(\varepsilon^{\gamma-\beta}|\psi(y)|\right)\,\d y
=\int_\R\frac12|\psi'(y)|^2+c\left(|\psi(y)|\right)\,\d y
\end{align*}
for the choice $\alpha=3\beta$ and $\beta=\gamma$; thus without loss of generality we shall from now on consider $\alpha=3$ and $\beta=\gamma=1$.
Apparently, for given mass flux magnitude $w$ the corresponding phase field attains the $\varepsilon$-independent profile $\psi$, scaled in height and width by $\frac1\varepsilon$ and $\varepsilon$, respectively.
Thus, the width of the diffused mass flux $\tilde\sigma$ concentrates more and more as $\tilde\sigma$ approaches $\sigma$ for $\varepsilon\to0$.

The optimal phase field $\tilde\sigma$ minimizes the phase field functional among all phase field functions carrying flux $w$ vertically upwards.
Thus we can summarize the above reasoning as
\begin{equation*}%\label{eqn:inverseProblem}
\tau(w)=\inf\left\{F^c[\psi]\,\middle|\,\int_\R\psi\,\d y=w\right\}
\quad\text{for }
F^c[\psi]=\int_\R\frac12|\psi'(y)|^2+c\left(|\psi(y)|\right)\,\d y\,.
\end{equation*}
If $c$ is chosen such that the above holds, $E^\tau_\varepsilon$ will be a valid phase field approximation of $E^\tau$,
thereby justifying inverse \cref{pbm:inverseProblem}.

\paragraph{Inverse problem for $c$.}
The aim of this article is to solve the above-described inverse \cref{pbm:inverseProblem} of finding a phase field cost $c:[0,\infty)\to[0,\infty)$ such that \eqref{eqn:inverseProblem} holds.
We here briefly show a heuristic argument how it relates to the linear deconvolution problem \eqref{eqn:convolutionProblem}.
This argument illustrates the basic intuition behind our approach which will be made rigorous in the subsequent sections.

It is straightforward to see that the optimal $\psi$ in \eqref{eqn:inverseProblem} is nonnegative.
Thus its optimality conditions read
\begin{align*}
0&=-\psi''+c'(\psi)+\lambda\,,\\
0&=\int_\R\psi\,\d y-w
\end{align*}
for some Lagrange multiplier $\lambda\in\R$.
Let $\psi_w:\R\to\R$ be the corresponding solution.
By differentiating $\tau(w)=F^c[\psi_w]$ with respect to $w$ we obtain
\begin{multline*}
\tau'(w)
=\partial_{\psi}F^c[\psi_w](\partial_w\psi_w)
=\int_\R\psi_w'\cdot\partial_w\psi_w'+c'(\psi)\partial_w\psi_w\,\d y\\
=\int_\R(-\psi_w''+c'(\psi))\partial_w\psi_w\,\d y
=-\lambda\int_\R\partial_w\psi_w\,\d y
=-\lambda\partial_w\int_\R\psi_w\,\d y
=-\lambda\,,
\end{multline*}
where we used an integration by parts as well as the optimality conditions.
Thus, $\psi_w$ actually solves the ordinary differential equation
\begin{equation}\label{eqn:optimalityCondition}
0=-\psi''+c'(\psi)-\tau'(w)\,.
\end{equation}
As a sideremark, note that this provides a different formulation of the inverse problem more akin to classical nonlinear parameter estimation in elliptic partial differential equations.
Indeed, introducing the forward operator $T:\psi\mapsto\int_\R\psi\,\d y$, we
\begin{gather*}
\text{seek the nonlinearity $c'$}\\
\text{such that the solution $\psi^r$ of the nonlinear elliptic equation $0=-\psi''+c'(\psi)-r$}\\
\text{satisfies $T(\psi^r)=(\tau')^{-1}(r)$ for all $r\in[0,\infty)$.}
\end{gather*}
In other words, for different spatially constant right-hand sides of the elliptic equation we measure the accumulated mass of the solution and have to find the nonlinearity such that the measurement fits to the given data.
Next we apply the Modica--Mortola trick which is standard in phase field methods.
Testing the optimality condition \eqref{eqn:optimalityCondition} with $2\psi_w'$ we obtain $0=\left(-|\psi_w'|^2+2(c(\psi_w)-\tau'(w)\psi_w)\right)'$.
Together with $\psi_w(y),\psi_w'(y)\to0$ as $|y|\to\infty$ and $c(0)=0$ this implies
\begin{equation*}
|\psi_w'|^2=2(z(\psi_w)-\tau'(w))\psi_w
\quad\text{for }
z(\phi)=\frac{c(\phi)}\phi\,.
\end{equation*}
In particular, letting without loss of generality $\psi_w$ achieve its maximum in $y=0$ so that $\psi_w'(0)=0$, we have
\begin{equation*}
0=2(z(\psi_w(0))-\tau'(w))\psi_w(0)\,.
\end{equation*}
By Young's inequality, $\frac{a^2}2+\frac{b^2}2\geq ab$ for all $a,b\in\R$ with equality if and only if $|a|=|b|$, we now obtain
\begin{align*}
\tau(w)-\tau'(w)w
&=F^c[\psi_w]-\tau'(w)\int_\R\psi_w\,\d y\\
&=\int_{-\infty}^0\frac{|\psi_w'|^2}2+\frac{2(z(\psi_w)-\tau'(w))\psi_w}2\,\d y+\int_0^{\infty}\frac{|\psi_w'|^2}2+\frac{2(z(\psi_w)-\tau'(w))\psi_w}2\,\d y\\
&=\int_{-\infty}^0|\psi_w'|\sqrt{2(z(\psi_w)-\tau'(w))\psi_w}\,\d y+\int_0^\infty|\psi_w'|\sqrt{2(z(\psi_w)-\tau'(w))\psi_w}\,\d y\\
&=2\int_0^{\psi_w(0)}\sqrt{2(z(\phi)-\tau'(w))\phi}\,\d\phi\,,
\end{align*}
where we performed the change of variables $\phi=\psi_w(y)$ and assumed that $\psi_w'$ does not change sign on $(-\infty,0)$ and $(0,\infty)$.
Noting $\psi_w(0)=z^{-1}(\tau'(w))$ and performing an integration by parts, we thus have
\begin{align*}
\tau(w)-\tau'(w)w
&=2\int_0^{z^{-1}(\tau'(w))}\sqrt{2(z(\phi)-\tau'(w))\phi}\,\d\phi\\
&=-\frac43\int_0^{z^{-1}(\tau'(w))}\phi^{\frac32}\frac{z'(\phi)}{\sqrt{2(z(\phi)-\tau'(w))}}\,\d\phi\\
&=-\frac43\int_{z(0)}^{\tau'(w)}z^{-1}(s)^{\frac32}\frac{1}{\sqrt{2(s-\tau'(w))}}\,\d s\,,
\end{align*}
where we have performed yet another change of variables $s=z(\phi)$.
Assuming now $z(0)=\infty$ and substituting $w=(\tau')^{-1}(t)$ we finally obtain
\begin{equation*}
[-\tau(-\cdot)]^\ast(t)
=\tau((\tau')^{-1}(t))-(\tau')^{-1}(t)t
=-\frac43\int_{\infty}^{t}\frac{z^{-1}(s)^{\frac32}}{\sqrt{2(s-t)}}\,\d s
=\int_t^\infty g(s)r(t-s)\,\d s
=\left[g*r\right](t)
\end{equation*}
for $r$ and $g$ as defined in \eqref{eqn:convFunction} and \eqref{eqn:nonlinearTransformation},
where $(\cdot)^\ast$ denotes the Legendre--Fenchel conjugate and $*$ convolution of two functions.
This is the linear deconvolution problem already mentioned in \eqref{eqn:convolutionProblem}.
Hence, given $\tau$ or equivalently $[-\tau(-\cdot)]^\ast$, we can first solve \eqref{eqn:convolutionProblem} for $g$
and then obtain
\begin{equation}\label{eqn:cFromG}
c(\psi)=\psi g^{-1}(\psi^{3/2})\,.
\end{equation}
This argument will be made rigorous in \cref{sec:deconvolution}.

\section{Examples and piecewise linear approximation}\label{sec:examples}
We begin by illustrating the use of our characterization of the inverse problem as a linear deconvolution problem by a few examples.
The examples will also later aid to prove existence of a solution to \cref{pbm:inverseProblem}, and they will furthermore illustrate the potential nonuniqueness of this solution.
%That the function $c$ derived in these examples is indeed correct will be justified in retrospect at the end of the article.

\begin{example}[Classical branched transport]\label{exm:branchedTransport}
The standard model of branched transport has a parameter $\alpha\in(0,1)$ and uses the cost
\begin{equation*}
\tau(w)=w^\alpha
\qquad\text{ so that }\quad
[-\tau(-\cdot)]^*(t)=(1-\alpha)\left(\frac t{\alpha}\right)^{\frac\alpha{\alpha-1}}\,.
\end{equation*}
In that case, \eqref{eqn:convolutionProblem} is solved by
\begin{align*}
g(s)=Ks^{-\beta}
\quad&\text{for }
\beta=\frac12+\frac\alpha{1-\alpha}\\
&\text{and }
K=\frac{1-\alpha}\alpha(\alpha)^{\frac1{1-\alpha}}\frac{2\beta-1}{2\sqrt{2\pi}}\frac{3\Gamma(\beta)}{2\Gamma(\beta+\frac12)}
=(\alpha)^{\frac1{1-\alpha}}\frac{1}{\sqrt{2\pi}}\frac{3\Gamma(\frac12+\frac\alpha{1-\alpha})}{2\Gamma(1+\frac\alpha{1-\alpha})}
\end{align*}
(indeed, this follows for instance from noting $\int_t^\infty s^{-\beta}/\sqrt{s-t}\,\d s=\frac2{2\beta-1}\frac{\Gamma(\beta+1/2)}{\Gamma(\beta)}\sqrt\pi t^{1/2-\beta}$).
Thus, by \eqref{eqn:nonlinearTransformation} we have
\begin{equation*}
c(\phi)
=\phi\left[\frac1{K}\phi^{3/2}\right]^{-2\frac{1-\alpha}{1+\alpha}}
=K^{2\frac{1-\alpha}{1+\alpha}}\phi^{2\frac{2\alpha-1}{1+\alpha}}\,.
\end{equation*}
This is exactly the model derived in \cite{OuSa11}.
A corresponding illustration is given in \cref{fig:examples}.
% \todo{Using the Laplace transform $\mathcal L$, one gets
% \begin{gather*}
% \hat r(\xi)
% =\int_\R r(t)e^{-2\pi i\xi t}\,\d t
% =\int_0^\infty\sqrt{\tfrac2t}e^{2\pi i\xi t}\,\d t
% =2^{3/2}\mathcal L(\tfrac1{2\sqrt t})(-2\pi i\xi)
% =2^{3/2}\Gamma(3/2)/\sqrt{-2\pi i\xi}\,,\\
% \widehat{(\cdot)^{-\beta}}(\xi)
% =\int_0^\infty t^{-\beta}e^{-2\pi i\xi t}\,\d t
% =\mathcal L((\cdot)^{-\beta})(2\pi i\xi)
% =\Gamma(1-\beta)(2\pi i\xi)^{\beta-1}\,,
% \end{gather*}
% however, this does not allow to have $\hat r\widehat{(\cdot)^{-\beta}}=\text{const}\widehat{(\cdot)^{-\gamma}}$.
% Can one fix this so that at least formally one finds $g$ via the Fourier transform?%
% }%\todo
\end{example}

\begin{figure}
\setlength\unitlength{.9\linewidth}
\hspace*{-.1\unitlength}\includegraphics[width=\unitlength]{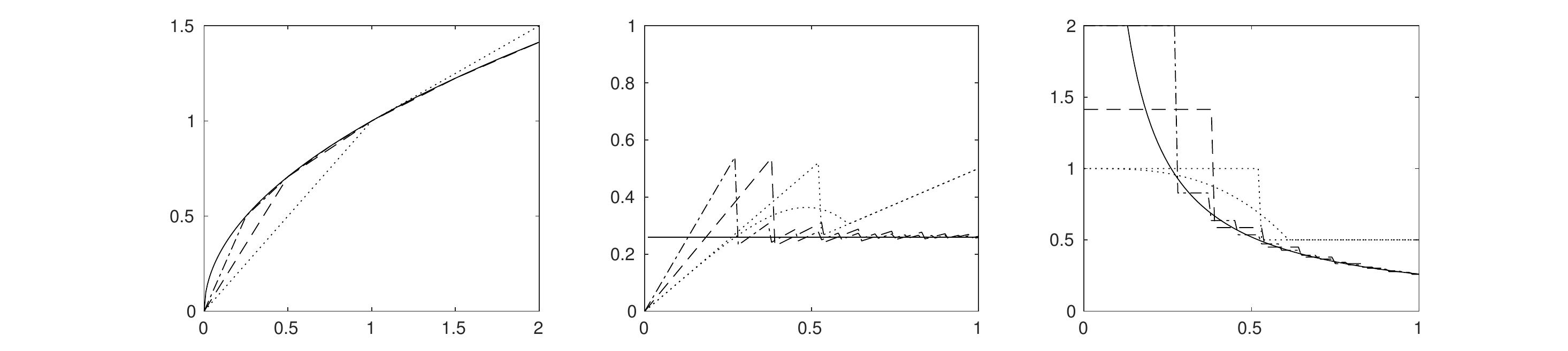}%obtained using piecewiseLinearTau.m
\begin{picture}(0,0)(1,0)
\put(.23,-.015){\small$w$}
\put(.51,-.015){\small$\phi$}
\put(.79,-.015){\small$\phi$}
\put(.09,.09){\rotatebox{90}{\small$\tau(w)$}}
\put(.37,.09){\rotatebox{90}{\small$c(\phi)$}}
\put(.65,.09){\rotatebox{90}{\small$z(\phi)$}}
\put(.92,.18){\small\makebox[5ex][l]{---} $\tau(w)=\sqrt{w}$}
\put(.92,.15){\small\makebox[5ex][l]{$\cdots$} $\tau(w)=\min\{w,(1+w)/2\}$}
\put(.92,.12){\small\makebox[5ex][l]{$-\,-$} $\tau(w)=I_{\frac12}[\sqrt\cdot](w)$}
\put(.92,.09){\small\makebox[5ex][l]{$-\cdot-$} $\tau(w)=I_{\frac14}[\sqrt\cdot](w)$}
\end{picture}%
\caption{Different transport costs $\tau$ and corresponding phase field costs $c$ as well as mass-specific phase field costs $z$ from \crefrange{exm:branchedTransport}{exm:urbanPlanningII} and \cref{thm:genUrbPln}.
$I_h[f]$ denotes the piecewise linear interpolation of a function $f$ with grid size $h$.
Note that for $\tau(w)=\min\{w,(1+w)/2\}$ two different corresponding phase field costs $c$ are shown.}
\label{fig:examples}
\end{figure}
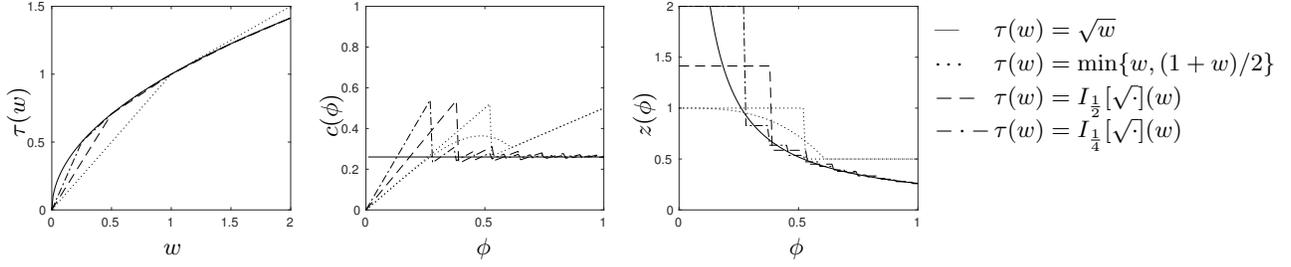

\begin{example}[Urban planning I]\label{exm:urbanPlanning}
Another branched transportation problem from the literature, so-called urban planning \cite{BuPrSoSt09,BrWi15-equivalent}, is given by 
\begin{equation*}
\tau(w)=\min(aw,bw+d)
\qquad\text{with }
[-\tau(-\cdot)]^*(t)=\begin{cases}
\infty&\text{if }t<b,\\
\frac d{a-b}(a-t)&\text{if }t\in[b,a),\\
0&\text{else}
\end{cases}
\end{equation*}
for parameters $a>b>0$, $d>0$.
In that case, \eqref{eqn:convolutionProblem} is solved by
\begin{equation*}
g(s)=\begin{cases}
\infty&\text{if }s<b,\\
\frac3{\sqrt2\pi}\frac d{a-b}\sqrt{a-s}&\text{if }s\in[b,a),\\
0&\text{else}
\end{cases}
\end{equation*}
(indeed, this follows for instance from noting $\int_t^a\sqrt{\frac{a-s}{s-t}}\,\d s=\frac\pi2(a-t)$).
Thus, by \eqref{eqn:nonlinearTransformation} (using the generalized inverse $g^{-1}(\phi)=\inf\{x\in[0,\infty)\,|\,z(x)\leq\phi\}$ of $g$, see \cref{thm:generalizedInverse} later) we have
\begin{equation*}
c(\phi)=\phi\left.\begin{cases}
a-\left(\frac{\sqrt2\pi}3\frac{a-b}d\phi^{\frac32}\right)^2&\text{if }\phi^{\frac32}\leq\frac3{\sqrt2\pi}\frac d{a-b}\sqrt{a-b},\\
b&\text{else}
\end{cases}\right\}
=\max\left\{a\phi-\tfrac{2\pi^2}9(\tfrac{d}{a-b})^2\phi^4,b\phi\right\}\,.
\end{equation*}
\end{example}

\begin{example}[Urban planning II]\label{exm:urbanPlanningII}
\Cref{thm:genUrbPln} will show that for the same $\tau$ as in \cref{exm:urbanPlanning} one can also use a simpler, piecewise constant $c$ inducing the same urban planning cost $\tau$,
\begin{equation*}
c(\phi)=\begin{cases}
a\phi&\text{if }\phi<P,\\
b\phi&\text{else}
\end{cases}
\quad\text{with }
P=\frac12\sqrt[3]{\frac{9d^2}{4(a-b)}}\,.
\end{equation*}
Note that the function $g$ corresponding to $c$ via \eqref{eqn:cFromG} is given by
\begin{equation*}
g(s)=\begin{cases}
\infty&\text{if }s<b,\\
P^{3/2}&\text{if }s\in[b,a),\\
0&\text{else,}
\end{cases}
\end{equation*}
which only satisfies \eqref{eqn:convolutionProblem} for $t=a$ and $t=b$.
This is in accordance with the previously mentioned fact that \eqref{eqn:convolutionProblem} is only a sufficient, but not a necessary condition (which will be proved in \cref{thm:sufficientCondition}).
In fact, by the necessary condition that will be provided in \cref{thm:necessaryCondition}, \eqref{eqn:convolutionProblem} has to hold only for $t\in\{\tau'(w)\,|\,w>0,\,\tau\text{ is differentiable in }w\}$,
which for urban planning amounts to $t\in\{a,b\}$.
\end{example}

The final example in this section is a generalization of urban planning, a piecewise affine transport cost $\tau$, which will serve two purposes.
First, any admissible transportation cost $\tau$ can be discretized with arbitrarily small error (in the supremum norm) by a piecewise linear interpolation $\tilde\tau$
so that the example provides a way to approximate any $\tau$ using a simple, explicit phase field cost belonging to a piecewise linear $\tilde\tau$ (\cref{fig:examples} shows an example).
Second, we will prove existence of a solution to \cref{pbm:inverseProblem} by approximation via piecewise linear models.
% Since the characterization \eqref{eqn:convolutionProblem} assumes existence of the phase field cost $c$ a priori,
% we cannot use it to derive the phase field cost corresponding to the generalization of urban planning.
% Therefore we present it as a separate theorem.
% Instead of working with the phase field cost $c$, below we will use the mass-specific phase field cost
Before the example, which will be presented in \cref{thm:genUrbPln}, we recall the notion of the mass-specific phase field cost
\begin{equation*}
z:[0,\infty)\to[0,\infty),\quad
z(\phi)=\frac{c(\phi)}\phi\,,
\end{equation*}
($z(0)$ can be defined as the right limit in $0$ using l'H\^opital's rule)
and state an auxiliary lemma which will help to simplify notation in several places.

\begin{lemma}[Phase field properties]\label{thm:phaseFieldProperties}
Given a phase field cost $c:[0,\infty)\to\R$ and a phase field function $\psi:\R\to[0,\infty)$ with finite phase field energy $F^c[\psi]<\infty$,
denote its symmetric decreasing rearrangement by $\tilde\psi$ (obviously, $\tilde\psi$ is an even function, monotonically decreasing on $[0,\infty)$).
Then $\tilde\psi$ is bounded and continuous with $F^c[\tilde\psi]\leq F^c[\psi]$.
\end{lemma}
\begin{proof}
By definition of the symmetric decreasing rearrangement we have $\int_\R c(\psi)\,\d y=\int_\R c(\tilde\psi)\,\d y$.
Furthermore, by the P\'olya--Szeg\"o inequality we have $\int_\R\frac12|\psi'|^2\,\d y\geq\int_\R\frac12|\tilde\psi'|^2\,\d y$, resulting in $F^c[\psi]\geq F^c[\tilde\psi]$, as desired.
Obviously, $\tilde\psi$ lies in the Hilbert space $H^1((-R,R))$ for all $R>0$ and thus is continuous on any interval $[-R,R]$ by Sobolev embedding.
\end{proof}

\begin{theorem}[Generalized urban planning]\label{thm:genUrbPln}
Let the mass-specific phase field cost $z$ be piecewise constant,
\begin{equation*}
z(\phi)=c(\phi)/\phi=a_i
\quad\text{for }\phi\in[\phi_{i-1},\phi_i),
\,i\in I,
\end{equation*}
where $I=\{1,\ldots,N\}$ or $I=\N$ and %$(\phi_i,a_i)_{i\in I}$ are coefficients satisfying
$0=\phi_0<\phi_1<\phi_2<\ldots$ as well as $\infty>a_1>a_2>\ldots$ with $\phi_N=\infty$ if $I$ is finite.
% Let $0=\phi_0\leq\phi_1\leq\phi_2\leq\ldots$ and $\infty>a_1\geq a_2\geq a_3\geq\ldots$ as well as
% \begin{equation*}
% z(\phi)=c(\phi)/\phi=a_i\text{ if }\phi\in[\phi_{i-1},\phi_i)\,.
% \end{equation*}
The cost $\tau$ induced via \eqref{eqn:inverseProblem} is piecewise affine and reads
\begin{equation}\label{eqn:genUrbPln}
\tau(w)=\inf\left\{wa_j+\frac{4\sqrt2}3\sum_{i=1}^{j-1}\sqrt{a_i-a_j}(\phi_i^{3/2}-\phi_{i-1}^{3/2})\,\middle|\,j\in I\right\}\,.
\end{equation}
Furthermore, for any given admissible piecewise affine $\tau$ one can determine coefficients $a_i,\phi_i$ from the above formula such that the corresponding phase field cost $c$ induces $\tau$ via \eqref{eqn:inverseProblem}.
\end{theorem}
\begin{proof}
To identify the induced $\tau$ we will explicitly construct optimal phase field profiles $\psi$ minimizing $F^c$ in \eqref{eqn:inverseProblem} for prescribed total mass $\int_\R\psi\,\d y$ and calculate their energy $F^c$.
These phase field profiles will be composed of multiple segments connecting the phase field values $\phi_{i-1}$ and $\phi_i$ for $i\in I$.
To identify these segments we first consider the auxiliary optimization problem
\begin{equation}\label{eqn:auxOptPbm}
E_{\phi_r,\phi_l}^w(T)=
\min\left\{\int_0^T|\varphi'|^2\,\d t\,\middle|\,\varphi(0)=\phi_l,\,\varphi(T)=\phi_r,\,\varphi([0,T])\subset[\phi_l,\phi_r],\,\int_0^T\varphi\,\d t=w\right\}
\end{equation}
for $\phi_r>\phi_l\geq0$, $T>0$, and $w\in[\phi_lT,\phi_rT]$ (a different $w$ would be incompatible with the constraints).

\emph{Step\,1.}
We explicitly solve \eqref{eqn:auxOptPbm}.
Since it is a convex optimization problem, by standard convex duality one obtains the optimality conditions
\begin{align*}
0&=\int_0^T\varphi\,\d t-w\,,\\
0&=\varphi''-\lambda\text{ on }(t_1,t_2)\,,\\
0&=\varphi-\phi_l\text{ on }[0,t_1]\,,\\
0&=\varphi-\phi_r\text{ on }[t_2,T]
\end{align*}
for some $0\leq t_1<t_2\leq T$ and a Lagrange multiplier $\lambda\in\R$
(that the box constraints are only active on intervals $[0,t_1]$ and $[t_2,T]$ follows again from the Polya--Szeg\"o inequality).
We first show $t_1=0$ or $t_2=T$.
Indeed, assume the opposite and set $\hat t(=(T\psi_r-w)/(\psi_r-\psi_l)\in(0,T)$,
then $\hat\varphi(t)=\varphi(\hat t+(1-\delta)(t-\hat t))$ for $\delta>0$ small enough still satisfies $\hat\varphi(0)=\psi_l$, $\hat\varphi(T)=\psi_r$, $\hat\varphi([0,T])\subset[\psi_l,\psi_r]$, and
\begin{equation*}
\int_0^T\hat\varphi(t)\,\d t
=\frac1{1-\delta}\int_{(1-\delta)\hat t}^{T-(1-\delta)(T-\hat t)}\varphi(s)\,\d s
=\frac{w-\psi_l(1-\delta)\hat t-\psi_r(1-\delta)(T-\hat t)}{1-\delta}
=w\,,
\end{equation*}
but $\hat\varphi$ has strictly smaller energy than $\varphi$, leading to a contradiction.

Now consider the case $t_2=T$ so that the optimal $\varphi$ is some cubic polynomial on $[t_1,T]$, that is,
\begin{equation*}
\varphi(t)=\begin{cases}
\phi_l&\text{if }t\leq t_1\\
\frac{t-t_1}{T-t_1}\phi_r+\frac{T-t}{T-t_1}\phi_l+\frac\lambda2(t-t_1)(t-T)&\text{else.}
\end{cases}
\end{equation*}
Note that $\varphi([0,T])\subset[\psi_l,\psi_r]$ is equivalent to $|\lambda|\leq2\frac{\phi_r-\phi_l}{(T-t_1)^2}$.
The condition $\int_0^T\varphi\,\d t=w$ then can be solved for $\lambda$ as
\begin{equation*}
\lambda=\lambda(t_1)=6\frac{(\phi_r-\phi_l)(T-t_1)+2(\phi_lT-w)}{(T-t_1)^3}\,.
\end{equation*}
Substituting $\varphi$ back into the Dirichlet energy we obtain
\begin{equation*}
\int_0^T|\varphi'|^2\,\d t
=3\frac{((\phi_r-\phi_l)(T-t_1)+2(\phi_lT-w))^2}{(T-t_1)^3}+\frac{(\phi_r-\phi_l)^2}{T-t_1}\,,
\end{equation*}
which is monotonically increasing in $t_1\in[0,T]$ as can be seen from the derivative with respect to $t_1$, given by $4((\phi_r-\phi_l)(T-t_1)+3(\phi_lT-w))^2/(T-t_1)^4$.
Thus, the Dirichlet energy is minimized for $t_1=0$  if this is admissible in the sense $|\lambda(t_1)|\leq2\frac{\phi_r-\phi_l}{(T-t_1)^2}$, which for $t_1=0$ is equivalent to $$w\in\left[\tfrac{2\phi_l+\phi_r}3T,\tfrac{\phi_l+2\phi_r}3T\right]\,.$$
Otherwise, we can distinguish two cases:
If $w>\frac{\phi_l+2\phi_r}3T$, then $\lambda(t_1)<-2\frac{\phi_r-\phi_l}{(T-t_1)^2}$ for all $t_1\in[0,T]$ so that there is no admissible $t_1$.
If $w<\frac{2\phi_l+\phi_r}3T$, then the smallest admissible $t_1$ can be calculated as
\begin{equation*}
\hat t_1=T-3\frac{w-\phi_lT}{\phi_r-\phi_l}\,.
\end{equation*}
Summarizing, if $t_2=T$ for the optimal $\varphi$, then $t_1=\max\{0,\hat t_1\}$ and we obtain
\begin{equation*}
E_{\phi_l,\phi_r}^w(T)
=\int_0^T|\varphi'|^2\,\d t
=\begin{cases}
3\frac{((\phi_r+\phi_l)T-2w)^2}{T^3}+\frac{(\phi_r-\phi_l)^2}{T}&\text{if }w\geq\frac{2\phi_l+\phi_r}3T\text{ (and thus $t_1=0$)}\,,\\
\frac49\frac{(\phi_r-\phi_l)^3}{w-\phi_lT}&\text{else (and thus $t_1=\hat t_1>0$)}\,.
\end{cases}
\end{equation*}
The second term is increasing in $T$ and is only valid for $T>\frac{3w}{2\phi_l+\phi_r}$; thus $E_{\phi_l,\phi_r}^w(T)$ must achieve its minimum in the first term, for $T\leq\frac{3w}{2\phi_l+\phi_r}$.
The first term is minimized by
\begin{equation*}
\hat T_{\phi_l,\phi_r}^w
%=3w\frac{\phi_l-\sqrt{\phi_l\phi_r}+\phi_r}{\phi_l^2+\phi_l\phi_r+\phi_r^2}
=\frac{3w}{\phi_l+\sqrt{\phi_l\phi_r}+\phi_r}
\qquad\text{with}\qquad
E_{\phi_l,\phi_r}^w(\hat T_{\phi_l,\phi_r}^w)
%=\frac49\frac{(\phi_l^2+\phi_l\phi_r+\phi_r^2)^2(-3(\phi_l+\phi_r)\sqrt{\phi_l\phi_r}+\phi_l^2+4\phi_l\phi_r+\phi_r^2)}{w(\phi_l-\sqrt{\phi_l\phi_r}+\phi_r)^3}
=\frac49\frac{(\phi_r^{3/2}-\phi_l^{3/2})^2}{w}
\end{equation*}
(it is decreasing for $T\leq\hat T_{\phi_l,\phi_r}^w$ and increasing for $T\in[\hat T_{\phi_l,\phi_r}^w,\frac{3w}{2\phi_l+\phi_r}]$).

Analogously one can consider the case $t_1=0$ and $t_2\leq T$, and one arrives at the same minimum cost and same optimal $\varphi$. Summarizing
(and abbreviating $\ell(\phi_l,\phi_r)=2(\phi_r^{3/2}-\phi_l^{3/2})^2/9$),
\begin{multline*}
\min\left\{\int_0^T|\varphi'|^2\,\d t\,\middle|\,\varphi(0)=\phi_l,\,\varphi(T)=\phi_r,\,\varphi\in[\phi_l,\phi_r],\,\int_0^T\varphi\,\d t=w,\,T>0\right\}\\
=\min\left\{E_{\phi_l,\phi_r}^w(T)\,\middle|\,T>0\right\}
=E_{\phi_l,\phi_r}^w(\hat T_{\phi_l,\phi_r}^w)
=\frac49\frac{(\phi_r^{3/2}-\phi_l^{3/2})^2}{w}
=2\frac{\ell(\phi_l,\phi_r)}w\,,
\end{multline*}
where the optimum is achieved by choosing
\begin{gather*}
T=\hat T_{\phi_l,\phi_r}^w\,,\quad
t_1=0\,,\quad
t_2=T\,,\quad
\lambda
%=\frac2{27}\frac{(\phi_l^2+\phi_l\phi_r+\phi_r^2)^2(-3(\phi_l+\phi_r)\sqrt{\phi_l\phi_r}+\phi_l^2+4\phi_l\phi_r+\phi_r^2)}{w^2(\phi_l-\sqrt{\phi_l\phi_r}+\phi_r)^3}
=\frac2{27}\frac{(\phi_r^{3/2}-\phi_l^{3/2})^2}{w^2}\geq0\,,\\
\varphi(t)=\varphi_{\phi_l,\phi_r}^w(t)=\frac tT\phi_r+\frac{T-t}T\phi_l+\frac\lambda2t(t-T)\,.
\end{gather*}
As a consequence, for any $a\geq0$ we also have
\begin{multline*}
\min\left\{\int_0^T\frac12|\varphi'|^2+a\varphi\,\d t\,\middle|\,\varphi(0)=\phi_l,\,\varphi(T)=\phi_r,\,\varphi\in[\phi_l,\phi_r],\,\int_0^T\varphi\,\d t=w,\,T>0\right\}
=\frac{\ell(\phi_l,\phi_r)}w+aw
\end{multline*}
with the same minimizer.

\emph{Step\,2.}
We show that $\tau$ is bounded below by the right-hand side of \eqref{eqn:genUrbPln}.
Hence, let $0=\phi_0<\phi_1<\ldots$ and $\infty>a_1>a_2>\ldots$ as well as $c(\phi)/\phi=a_i$ for $\phi\in[\phi_{i-1},\phi_i)$ and consider a function $\psi:\R\to[0,\infty)$ with $\int_\R\psi\,\d y=w$.
We would like to bound $F^c[\psi]$ from below.
By \cref{thm:phaseFieldProperties} we may assume without loss of generality that $\psi$ achieves its maximum $\hat\psi=\max\{\psi(y)\,|\,y\in\R\}$ in $y=0$ and is even and decreasing on $[0,\infty)$.
Next define $y_0=-\infty$ as well as
\begin{equation*}
y_i=\min\{y\in\R\,|\,\psi(y)\geq\phi_i\}\leq0
\quad\text{ and }\quad
w_i=\int_{y_{i-1}}^{y_i}\psi\,\d y\in[0,w]
\end{equation*}
for $i=1,\ldots,j=\max\{k\in I\,|\,\phi_k\leq\hat\psi\}$, and $w_{j+1}=\int_{y_j}^{-y_j}\psi\,\d y$.
Using the result of the previous step we have
\begin{align*}
w&=2\sum_{i=1}^{j}w_i+w_{j+1}\,,\\
F^c[\psi]&=2\sum_{i=1}^{j}\int_{y_{i-1}}^{y_i}\frac12|\psi'|^2+a_i\psi\,\d y+\int_{y_j}^{-y_j}\frac12|\psi'|^2+a_{j+1}\psi\,\d y\\
&\geq2\sum_{i=1}^{j}\left(\frac{\ell(\phi_{i-1},\phi_i)}{w_i}+a_iw_i\right)+a_{j+1}w_{j+1}\\
&=a_{j+1}w+2\sum_{i=1}^{j}\left(\frac{\ell(\phi_{i-1},\phi_i)}{w_i}+(a_i-a_{j+1})w_i\right)\\
&\geq a_{j+1}w+2\sum_{i=1}^{j}2\sqrt{\ell(\phi_{i-1},\phi_i)}\sqrt{a_i-a_{j+1}}\,,
\end{align*}
where in the last step we have minimized for the $w_i$, yielding
\begin{equation*}
w_i=W_{i,j+1}=\sqrt{\frac{\ell(\phi_{i-1},\phi_i)}{a_i-a_{j+1}}}\,.
\end{equation*}
Since $\psi$ was arbitrary, we obtain
\begin{equation*}
\tau(w)=\inf\left\{F^c[\psi]\,\middle|\,\int_\R\psi\,\d y=w\right\}
\geq\inf\left\{a_{j}w+4\sum_{i=1}^{j-1}\sqrt{\ell(\phi_{i-1},\phi_i)}\sqrt{a_i-a_{j}}\,\middle|\,j\in I\right\}\,.
\end{equation*}

\emph{Step\,3.}
We show that $\tau$ is also bounded above by the right-hand side of \eqref{eqn:genUrbPln}, that is,
that the previous inequality actually is an equality.
To this end, we show by induction in $j\in I$ that for any $w>0$ and $\delta>0$ we can find some $\psi:\R\to[0,\infty)$ with $\int_\R\psi\,\d y=w$ and $F^c[\psi]\leq F_j+\delta$
for $F_j=a_{j}w+4\sum_{i=1}^{j-1}\sqrt{\ell(\phi_{i-1},\phi_i)}\sqrt{a_i-a_{j}}$.
First consider the case $j=1$: it is straightforward to check that $\psi(y)=w\eta\max\{0,1-\eta|y|\}$ with $\eta=\sqrt[3]{\delta/w^2}$ satisfies $F^c[\psi]\leq a_1w+\delta$.
Now assume the induction hypothesis to hold for $j-1\geq1$; we aim to show existence of $\psi:\R\to[0,\infty)$ with $\int_\R\psi\,\d y=w$ and $F^c[\psi]\leq F_j+\delta$.
If $2\sum_{i=1}^{j-1}W_{ij}\leq w$, we can recursively define $y_j=-(w/2-\sum_{i=1}^{j-1}W_{ij})/\phi_{j-1}$ and $y_{i-1}=y_i-\hat T_{\phi_{i-1},\phi_i}^{W_{ij}}$ for $i=j,j-1,\ldots,1$ and choose
\begin{equation*}
\psi(y)=\begin{cases}
\phi_{j-1}&\text{if }|y|\leq|y_j|,\\
\varphi_{\phi_{i-1},\phi_i}^{W_{ij}}(|y_{i-1}|-|y|)&\text{if }|y|\in[|y_i|,|y_{i-1}|],\\
0&\text{else}
\end{cases}
\end{equation*}
to obtain $\int_\R\psi\,\d y=w$ and $F^c[\psi]=F_j$ (the function $\psi$ is illustrated in \cref{fig:generalizedUrbanPlanning}).
If on the other hand $2\sum_{i=1}^{j-1}W_{ij}>w$, then
\begin{align*}
F_j=
a_{j}w+4\sum_{i=1}^{j-1}\sqrt{\ell(\phi_{i-1},\phi_i)}\sqrt{a_i-a_{j}}
&=a_{j}w+2\sum_{i=1}^{j-1}\left(\frac{\ell(\phi_{i-1},\phi_i)}{W_{ij}}+(a_i-a_{j})W_{ij}\right)\\
&=a_j\left(w-2\sum_{i=1}^{j-1}W_{ij}\right)+2\sum_{i=1}^{j-1}\left(\frac{\ell(\phi_{i-1},\phi_i)}{W_{ij}}+a_iW_{ij}\right)\\
&>a_{j-1}\left(w-2\sum_{i=1}^{j-1}W_{ij}\right)+2\sum_{i=1}^{j-1}\left(\frac{\ell(\phi_{i-1},\phi_i)}{W_{ij}}+a_iW_{ij}\right)\\
&=a_{j-1}w+2\sum_{i=1}^{j-1}\left(\frac{\ell(\phi_{i-1},\phi_i)}{W_{ij}}+(a_i-a_{j-1})W_{ij}\right)\\
&>a_{j-1}w+2\sum_{i=1}^{j-2}\left(\frac{\ell(\phi_{i-1},\phi_i)}{W_{ij}}+(a_i-a_{j-1})W_{ij}\right)\\
&\geq a_{j-1}w+2\sum_{i=1}^{j-2}\left(\frac{\ell(\phi_{i-1},\phi_i)}{W_{i,j-1}}+(a_i-a_{j-1})W_{i,j-1}\right)\\
&=a_{j-1}w+4\sum_{i=1}^{j-2}\sqrt{\ell(\phi_{i-1},\phi_i)}\sqrt{a_i-a_{j-1}}
=F_{j-1}\,.
\end{align*}
Therefore, letting $\psi:\R\to[0,\infty)$ be the phase field from the previous induction step with $\int_\R\psi\,\d y=w$ and $F^c[\psi]\leq F_{j-1}+\delta$,
we also find $F^c[\psi]\leq F_j+\delta$, which conlcudes the proof by induction.

\emph{Step\,4.}
We show how the coefficients $a_i$ and $\psi_i$ can be recovered from a given piecewise affine $\tau$ such that the corresponding phase field cost $c$ induces $\tau$.
To this end we simply take $a_1,a_2,\ldots$ to be the slopes of the linear segments of $\tau$ in decreasing order.
The $\phi_j$ are then calculated from the points $w_j$ at which the segment of slope $a_j$ meets with the segment of slope $a_{j+1}$.
If $\tau$ can be expressed as \eqref{eqn:genUrbPln}, then necessarily
\begin{equation*}
w_j
=\frac{4\sum_{i=1}^{j}\sqrt{\ell(\phi_{i-1},\phi_i)}\sqrt{a_i-a_{j+1}}-4\sum_{i=1}^{j-1}\sqrt{\ell(\phi_{i-1},\phi_i)}\sqrt{a_i-a_{j}}}{a_j-a_{j+1}}
=\frac{4\sqrt2}3\sum_{i=1}^j\frac{\phi_{i}^{3/2}-\phi_{i-1}^{3/2}}{\sqrt{a_i-a_{j+1}}+\sqrt{a_i-a_j}}
\end{equation*}
so that
\begin{equation*}
\phi_{j}^{3/2}-\phi_{j-1}^{3/2}
=\left(\frac{3w_j}{4\sqrt2}-\sum_{i=1}^{j-1}\frac{\phi_{i}^{3/2}-\phi_{i-1}^{3/2}}{\sqrt{a_i-a_{j+1}}+\sqrt{a_i-a_j}}\right)\sqrt{a_j-a_{j+1}}\,,
\end{equation*}
which can readily be solved for $\phi_j$ given $\phi_0,\ldots,\phi_{j-1}$
(note that the parentheses are no smaller than $\frac3{4\sqrt2}(w_j-w_{j-1})$ and thus positive).
\end{proof}

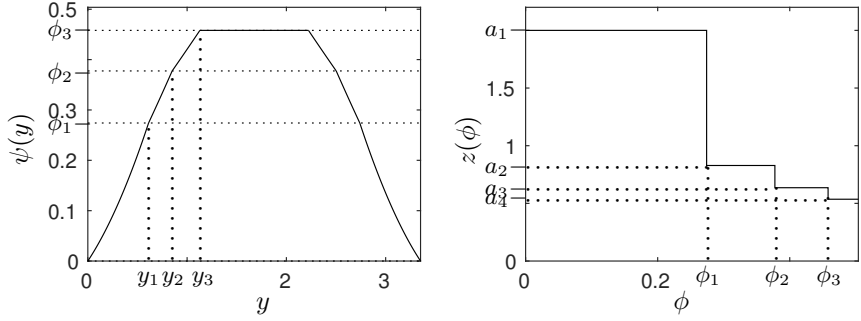
\begin{figure}
\centering
\setlength\unitlength{.8\linewidth}
\includegraphics[width=\unitlength]{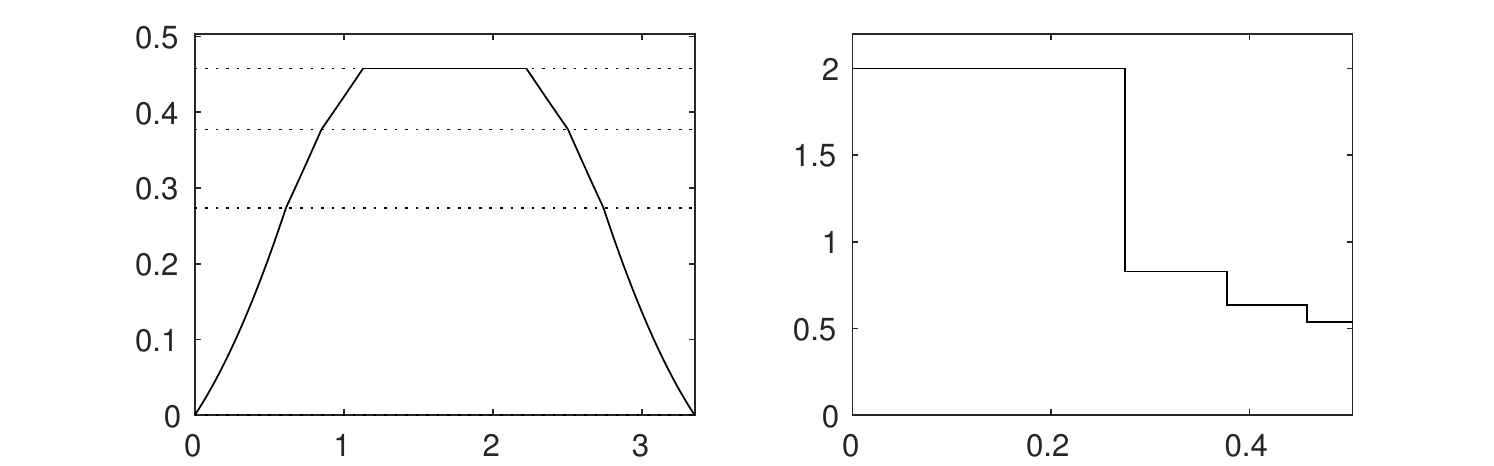}%obtained using piecewiseLinearTau.m
\begin{picture}(0,0)(1,0)
\put(.3,-.015){$y$}
\put(.05,.13){\rotatebox{90}{$\psi(y)$}}
\put(.09,.165){\small$\phi_1-$}
\put(.09,.217){\color{white}\rule{3ex}{4ex}}
\put(.09,.217){\small$\phi_2-$}
\put(.09,.26){\small$\phi_3-$}
\put(.72,-.015){$\phi$}
\put(.5,.13){\rotatebox{90}{$z(\phi)$}}
\multiput(.57,.121)(.01,0){19}{$\cdot$}
\multiput(.57,.099)(.01,0){25}{$\cdot$}
\multiput(.57,.088)(.01,0){31}{$\cdot$}
\put(.53,.07){\color{white}\rule{3ex}{4ex}}
\put(.53,.25){\color{white}\rule{3ex}{4ex}}
\put(.53,.26){\small$a_1-$}
\put(.53,.122){\small$a_2-$}
\put(.53,.1){\small$a_3-$}
\put(.53,.091){\small$a_4-$}
\multiput(.751,.03)(0,.01){10}{$\cdot$}
\multiput(.82,.03)(0,.01){8}{$\cdot$}
\multiput(.872,.03)(0,.01){7}{$\cdot$}
\put(.81,.0){\color{white}\rule{4ex}{2ex}}
\put(.742,.012){\small$\phi_1$}
\put(.812,.012){\small$\phi_2$}
\put(.864,.012){\small$\phi_3$}
\multiput(.188,.03)(0,.01){14}{$\cdot$}
\multiput(.212,.03)(0,.01){19}{$\cdot$}
\multiput(.24,.03)(0,.01){23}{$\cdot$}
\put(.21,.0){\color{white}\rule{4ex}{2ex}}
\put(.18,.012){\small$y_1$}
\put(.205,.012){\small$y_2$}
\put(.235,.012){\small$y_3$}
\end{picture}%
\caption{Optimal, piecewise quadratic phase field profile $\psi$ of fixed mass $w$ a for piecewise constant mass-specific phase field cost $z$.
The dotted lines indicate $\phi_1,\ldots,\phi_3$.}
\label{fig:generalizedUrbanPlanning}
\end{figure}

\section{Existence and properties of the phase field cost}\label{sec:existence}
In this section we show that for admissible transport costs $\tau$ there exists a (not necessarily unique) solution to \cref{pbm:inverseProblem}.
We will further present some a priori estimates on $c$.
The existence result will be based on approximating $\tau$ by its piecewise affine interpolation, reducing the problem to \cref{thm:genUrbPln},
and on the following simple \namecref{thm:monotonicityProperty}.

\begin{lemma}[Monotonicity property]\label{thm:monotonicityProperty}
Let transport costs $\tau$ and $\tilde\tau$ be induced via \eqref{eqn:inverseProblem} by $c$ and $\tilde c$, respectively.
If $c\leq\tilde c$, then $\tau\leq\tilde\tau$.
\end{lemma}
\begin{proof}
This follows immediately from $F^c\leq F^{\tilde c}$.
\end{proof}

Before the existence result (\cref{thm:existence}) we first provide some auxiliary lemmas.
The energy associated with a phase field function can be estimated via the standard Modica--Mortola trick, which we recall here.

\begin{lemma}[Modica--Mortola estimate]\label{thm:ModicaMortolaTrick}
Let $c:[0,\infty)\to[0,\infty)$ be lower semi-continuous and bounded on any interval $[a,b]\subset(0,\infty)$.
For any $T\in[0,\infty]$, $\hat\psi>0$, and $\psi:[0,T]\to[0,\infty)$ with $\psi(0)=\hat\psi$ and $\lim_{t\to T}\psi(t)=0$ we have
\begin{equation*}
\int_0^T\frac12|\psi'|^2+c(\psi)\,\d y\geq\int_0^{\hat\psi}\sqrt{2c(\phi)}\,\d\phi\,.
\end{equation*}
Furthermore, for any $\delta,\hat\psi>0$ there exist $T\in(0,\infty)$ and $\psi:[0,T]\to[0,\infty)$ with $\psi(0)=\hat\psi$ and $\psi(T)=0$ and
\begin{equation*}
\int_0^T\frac12|\psi'|^2+c(\psi)\,\d y\leq\int_0^{\hat\psi}\sqrt{2c(\phi)}\,\d\phi+\delta\,.
\end{equation*}
\end{lemma}
\begin{proof}
Using Young's inequality $a^2/2+b^2/2\geq ab$ one has
\begin{equation*}
\int_0^T\frac12|\psi'|^2+c(\psi)\,\d y
\geq\int_0^T|\psi'|\sqrt{2c(\psi)}\,\d y
\geq\int_0^{\hat\psi}\sqrt{2c(\phi)}\,\d\phi\,,
\end{equation*}
where in the last step we performed the change of variables $\phi=\psi(y)$.

To show the other inequality, define $\tilde\delta=\delta/\hat\psi$, $\Theta_{\max}=\int_0^\infty1/(\sqrt{c(\phi)}+\tilde\delta)\,\d\phi$, as well as the function
\begin{equation*}
\Theta:[0,\infty)\to[0,\Theta_{\max})\,,\quad
\Theta(\hat\phi)=\int_0^{\hat\phi}\frac1{\sqrt{c(\phi)}+\tilde\delta}\,\d\phi\,.
\end{equation*}
It is straightforward to see that $\Theta$ is Lipschitz continuous and invertible.
Indeed, $\Theta$ is the integral of an integrand which takes values in $(0,1/\tilde\delta]$,
thus it is in the Sobolev space $W^{1,\infty}$, is strictly increasing, and satisfies $\Theta(0)=0$.
Now let $T=\Theta(\hat\psi)$ and $\psi:[0,T]\to[0,\infty)$, $\psi(t)=\Theta^{-1}(T-t)$.
The function $\psi$ is locally Lipschitz on $(0,T)$.
Indeed, for $T>t_2>t_1>0$ we have
\begin{multline*}
\left|\frac{\psi(t_2)-\psi(t_1)}{t_2-t_1}\right|
=\left|\frac{\psi(t_2)-\psi(t_1)}{\Theta(T-\psi(t_2))-\Theta(T-\psi(t_1))}\right|\\
=\frac{(T-\psi(t_2))-(T-\psi(t_1))}{\int_{T-\psi(t_1)}^{T-\psi(t_2)}\frac1{\sqrt{c(\phi)}+\tilde\delta}\,\d\phi}
\leq\sqrt{\sup\{c(\phi)\,|\,\psi(t_2)\leq T-\phi\leq\psi(t_1)\}}+\tilde\delta\,.
\end{multline*}
Consequently, $\psi$ is differentiable almost everywhere with
\begin{equation*}
\psi'(t)=\frac{-1}{\Theta'(\psi(t))}=-(\sqrt{c(\psi(t))}+\tilde\delta)\,.
\end{equation*}
Thus, again using the change of variables $\phi=\psi(y)$, we obtain
\begin{multline*}
\int_0^T\frac12|\psi'|^2+c(\psi)\,\d y
\leq\int_0^T\frac12|\psi'|^2+\frac{(\sqrt{2c(\psi)}+\tilde\delta)^2}2\,\d y
=\int_0^T|\psi'|(\sqrt{2c(\psi)}+\tilde\delta)\,\d y\\
=\int_0^{\hat\psi}\sqrt{2c(\phi)}+\tilde\delta\,\d\phi
=\int_0^{\hat\psi}\sqrt{2c(\phi)}\,\d\phi+\delta\,.
\qedhere
\end{multline*}
\end{proof}

\begin{lemma}[Rescaling of a converging sequence]\label{thm:rescaling}
Let $z_n:[0,\infty)\to[0,\infty)$, $n=1,2,\ldots$, be a sequence of nonincreasing functions converging pointwise almost everywhere to some $z:[0,\infty)\to[0,\infty)$, and let $\delta>0$ as well as $\alpha,\beta>1$.
If $n$ is large enough, then for all $\phi\in(\delta,1/\delta)$ we have
\begin{equation*}
\tfrac1\alpha z_n(\beta\phi)-\delta\leq z(\phi)\leq\alpha z_n(\tfrac1\beta\phi)\,.
\end{equation*}
\end{lemma}
\begin{proof}
First note that $z$ also is nonincreasing as the limit of nonincreasing functions.
We first prove the second inequality.
Let $s=\sup\{\phi\in(\delta,1/\delta)\,|\,z(\phi)>0\}$, then $z>0$ on $(\delta,s)$.
Note that $z_n/z$ converges pointwise almost everywhere to $1$ on $(0,s)$ as $n\to\infty$.
Thus, by Egorov's theorem there exists some measurable $A\subset(0,s)$ with $(0,s)\setminus A$ having Lebesgue measure smaller than $\zeta=\delta(1-1/\beta)$ such that $z_n/z\to1$ uniformly on $A$.
Now pick $N\in\N$ such that $z_n(\phi)/z(\phi)>1/\alpha$ for all $\phi\in A$ and $n\geq N$.
Consider an arbitrary $n\geq N$.
For $\phi>s$ we trivially have
\begin{equation*}
z(\phi)=0\leq\alpha z_n(\tfrac1\beta\phi)\,.
\end{equation*}
Similarly, for $\phi\in A$ we have
\begin{equation*}
z(\phi)
\leq\alpha z_n(\phi)
\leq\alpha z_n(\tfrac1\beta\phi)\,.
\end{equation*}
Finally, let $\phi\in(\delta,s]\setminus A$, then there exists some $\hat\phi\in A\cap(\phi-\zeta,\phi)$ so that
\begin{equation*}
z(\phi)
\leq z(\hat\phi)
\leq\alpha z_n(\hat\phi)
\leq\alpha z_n(\tfrac1\beta\phi)
\end{equation*}
due to the monotonicity of $z$ and $z_n$ as well as $\hat\phi>\phi/\beta$.

The first inequality is shown similarly.
Indeed, since $B_n=(z_n(\phi)+\delta)/(z(\phi)+\delta)$ converges to $1$ for almost all $\phi\geq0$,
by Egorov's theorem there is some $A\subset(0,1/\delta+\delta)$ with $(0,1/\delta+\delta)\setminus A$ having Lebesgue measure smaller than $\zeta=\min\{\delta,\delta(\beta-1)\}$ such that $B_n$ converges uniformly on $A$.
Now pick $N\in\N$ such that $B_n<\alpha$ for all $\phi\in A$ and $n\geq N$, and consider an arbitrary $n\geq N$.
For any $\phi\in(\delta,1/\delta)$ there exists some $\hat\phi\in A\cap[\phi,\phi+\zeta)$ so that
\begin{equation*}
z(\phi)+\delta
\geq z(\hat\phi)+\delta
\geq\tfrac1\alpha(z_n(\hat\phi)+\delta)
\geq\tfrac1\alpha(z_n(\beta\phi)+\delta)
\geq\tfrac1\alpha z_n(\beta\phi)\,.
\qedhere
\end{equation*}
\end{proof}

\begin{remark}[Tighter rescaling bounds]
The lower bound of the previous lemma can be sharpened in different ways,
for instance its validity can be extended to all of $(\delta,\infty)$.
Also, the $\delta$ in the lower bound is only required if $z$ is not bounded away from zero.
However, for our purposes the above form of the statement is sufficient.
\end{remark}

\begin{lemma}[A priori estimate on phase field energy]\label{thm:energyEstimate}
Let $c:[0,\infty)\to[0,\infty)$ be arbitrary and let the phase field function $\psi:\R\to[0,\infty)$ have mass $\int_\R\psi\,\d y=w$ and maximum value $\hat\psi$.
Then $F^{c}[\psi]\geq\frac{(\hat\psi/2)^3}{w/2}$.
\end{lemma}
\begin{proof}
By \cref{thm:phaseFieldProperties} we may assume without loss of generality that $\psi$ achieves its maximum $\hat\psi$ in $y=0$ and is even and decreasing on $[0,\infty)$.
Letting $T=\max\{t>0\,|\,\psi(t)\geq\hat\psi/2\}$ we have
\begin{equation*}
w=\int_\R\psi\,\d y\geq\int_{-T}^T\psi\,\d y\geq2T\tfrac{\hat\psi}2=\hat\psi T
\end{equation*}
and thus, using Jensen's inequality,
\begin{equation*}
F^c[\psi]
\geq\int_{-T}^T\frac12|\psi'|^2\,\d y
=\int_0^T|\psi'|^2\,\d y
\geq T\left(\frac1T\int_0^T\psi'\,\d y\right)^2
=\frac{(\hat\psi/2)^2}T
\geq2\frac{(\hat\psi/2)^3}w\,.
\qedhere
\end{equation*}
\end{proof}

To simplify notation, in the following we abbreviate
\begin{equation*}
\tau^z(w)=\inf\left\{\{F^c[\psi]\,|\,\int_\R\psi\,\d y=w\right\}
\qquad\text{for }c(\phi)=\phi z(\phi)\,,
\end{equation*}
that is, $\tau^z$ is the transport cost induced via \eqref{eqn:inverseProblem}.

\begin{theorem}[Existence of phase field cost]\label{thm:existence}
\Cref{pbm:inverseProblem} has a solution $c:[0,\infty)\to[0,\infty)$ for every admissible transport cost $\tau$.
Furthermore, $c$ can be chosen such that the mass-specific phase field cost $\phi\mapsto z(\phi)=c(\phi)/\phi$ is lower semi-continuous and nonincreasing.
\end{theorem}
\begin{proof}
For $n=1,2,\ldots$ we approximate $\tau$ by a piecewise affine transport cost $\tau_n\leq\tau$ with increasing approximation quality, for instance the piecewise affine interpolation
\begin{equation*}
\tau_n(w)=2^n\left[\left(w-\tfrac{i-1}{2^n}\right)\tau\left(\tfrac{i}{2^n}\right)+\left(\tfrac{i}{2^n}-w\right)\tau\left(\tfrac{i-1}{2^n}\right)\right]
\qquad\text{for }
w\in\left[\tfrac{i-1}{2^n},\tfrac{i}{2^n}\right],\,
i=1,2,\ldots
\end{equation*}
% \begin{equation*}
% \tau_n(w)=\tau'(i2^{-n})w+b_i
% \qquad\text{if }w\in[(i-1)2^{-n},i2^{-n}],
% \quad i=1,2,\ldots
% \end{equation*}
% with $b_1=0$ and $b_2,b_3,\ldots$ chosen such that the $i$th and $(i+1)$th line segment intersect at $w=i2^{-n}$, thus $b_i = (a_{i-1}-a_i)i2^{-n}+b_{i-1}$.
% \todo{%
% Apparently, $z_n$ converges pointwise monotonically to some $z$, $z_n\nearrow z$.
% }%\todo
% Another alternative is just piecewise linear interpolation of $\tau$ at points $i2^{-n}$.
Due to \cref{thm:genUrbPln}, this piecewise affine $\tau_n$ is induced via \eqref{eqn:inverseProblem} by a lower semi-continuous phase field cost $c_n$ with nonincreasing piecewise constant mass-specific phase field cost $z_n(\phi)=c_n(\phi)/\phi$.
The proof now proceeds in steps.

\emph{Step\,1.}
The sequence of functions $z_n$ converges (up to a subsequence) pointwise almost everywhere to a nonincreasing lower semi-continuous function $z$.
To show this, note that for any $\delta>0$ the functions $z_n$ are uniformly bounded on $[\delta,\frac1\delta)$.
Indeed, assume the opposite, then due to the monotonicity of $z_n$, for any $a>0$ one can find some $n$ such that $z_n\geq\hat z$ for the function $\hat z:[0,\infty)\to[0,\infty)$ with $\hat z(\phi)=a$ for $\phi\in[0,\delta)$ and $\hat z(\phi)=0$ else.
Thus with \cref{thm:monotonicityProperty} and \cref{thm:genUrbPln} we obtain
\begin{equation*}
\tau_n(w)
=\tau^{z_n}(w)
\geq\tau^{\hat z}(w)
% =\inf\left\{\int_\R\frac12|\psi'|^2+z_n(\psi)\psi\,\d x\,\middle|\,\int_\R\psi\,\d x=w\right\}\\
% \geq\inf\left\{\int_\R\frac12|\psi'|^2\,\d x+\int_{\psi\leq\delta}a\psi\,\d x\,\middle|\,\int_\R\psi\,\d x=w\right\}
=\min\{aw,4\sqrt{2a\delta^3}/3\}
\quad\text{for all }w\geq0\,,
\end{equation*}
a contradiction for $a$ large enough.
Exploiting again the monotonicity of the $z_n$, it follows that these functions are actually even uniformly bounded in $\BV((\delta,\frac1\delta))$, the space of functions with bounded variation,
so that a subsequence converges weakly-* in $\BV((\delta,\frac1\delta))$.
Upon extracting yet another subsequence we thus obtain pointwise convergence almost everywhere.
%Thus, another subsequence converges strongly in any $L^p$, $p>0$, or even pointwise almost everywhere.
Since $\delta>0$ was arbitrary, by a standard diagonal argument we obtain a subsequence converging almost everywhere to some $z:[0,\infty)\to[0,\infty)$.
The monotonicity of the $z_n$ now implies monotonicity of $z$, and a monotonous $\BV$-function differs from its lower semi-continuous envelope at most on a nullset so that $z$ may be assumed lower semi-continuous.
Below, the index $n$ always refers to the extracted subsequence.
The remainder of the proof shows $\tau=\tau^z$.

\emph{Step\,2.}
We will need the following property of $z$ (which requires the continuity of $\tau$),
\begin{equation*}
\tau'(0)=\infty
\qquad\text{implies}\qquad
\lim_{\phi\to0}z(\phi)=\infty\,.
\end{equation*}
Indeed, we show that for any $a>0$ there is some $\phi_a>0$ such that $z_n(\phi)\geq a$ for all $\phi\in(0,\phi_a)$ and $n$ large enough.
To this end first note that not only $\tau_n\to\tau$ in the supremum norm, but also that the right derivative $\tau_n'$ converges pointwise to the right derivative $\tau'$
(the choice of the right derivative is just for notational convenience; one could likewise work with the left derivative or the full superdifferential).
Now pick $w_r>w_l>0$ such that $\tau_n'(w_r)\geq a$ and $\tau_n'(w_l)\geq2\tau_n'(w_r)$ for all $n$ large enough (which is possible due to $\tau'(0)=\infty$ and the continuity of $\tau$ in $0$).
Denote the coefficients of $z_n$ from \cref{thm:genUrbPln} by $a_{i,n}$ and $\phi_{i,n}$ and let $k_n$ and $j_n$ be the indices such that $a_{k_n,n}=\tau_n'(w_l)$ and $a_{j_n,n}=\tau_n'(w_r)$.
We can then estimate
\begin{align*}
&\frac{4\sqrt2}3\sqrt{\tau_n'(w_l)-\tau_n'(w_r)}\phi_{{j_n}-1,n}^{3/2}\\
&=\frac{4\sqrt2}3\sqrt{a_{k_n,n}-a_{j_n,n}}\sum_{i=1}^{{j_n}-1}(\phi_{i,n}^{3/2}-\phi_{i-1,n}^{3/2})\\
&\geq\frac{4\sqrt2}3\sum_{i=1}^{{k_n}-1}\left(\sqrt{a_{i,n}-a_{j_n,n}}-\sqrt{a_{i,n}-a_{k_n,n}}\right)(\phi_{i,n}^{3/2}-\phi_{i-1,n}^{3/2})+\frac{4\sqrt2}3\sum_{i={k_n}}^{{j_n}-1}\sqrt{a_{i,n}-a_{j_n,n}}(\phi_{i,n}^{3/2}-\phi_{i-1,n}^{3/2})\\
&=\frac{4\sqrt2}3\sum_{i=1}^{{j_n}-1}\sqrt{a_{i,n}-a_{j_n,n}}(\phi_{i,n}^{3/2}-\phi_{i-1,n}^{3/2})-\frac{4\sqrt2}3\sum_{i=1}^{{k_n}-1}\sqrt{a_{i,n}-a_{k_n,n}}(\phi_{i,n}^{3/2}-\phi_{i-1,n}^{3/2})\\
&=(\tau_n(w_r)-\tau_n'(w_r)w_r) - (\tau_n(w_l)-\tau_n'(w_l)w_l)\,,
\end{align*}
where we used the subadditivity $\sqrt{\alpha}-\sqrt{\beta}\leq\sqrt{\alpha-\beta}$ for all $\alpha\geq\beta\geq0$.
This implies
\begin{multline*}
\phi_{j_n-1,n}^{3/2}\geq\frac3{4\sqrt2}\frac{(\tau_n(w_r)-\tau_n'(w_r)w_r) - (\tau_n(w_l)-\tau_n'(w_l)w_l)}{\sqrt{\tau_n'(w_l)-\tau_n'(w_r)}}
\longrightarrow2\phi_a^{3/2}
\quad\text{as }n\to\infty\\
\text{for }\phi_a=\left(\frac3{8\sqrt2}\frac{(\tau(w_r)-\tau'(w_r)w_r) - (\tau(w_l)-\tau'(w_l)w_l)}{\sqrt{\tau'(w_l)-\tau'(w_r)}}\right)^{2/3}>0\,.
\end{multline*}
Thus, for all $n$ large enough we have
\begin{equation*}
\inf\{z_n(\phi)\,|\,\phi\leq\phi_a\}
\geq z_n(\phi_a)
\geq z_n(\phi_{j_n-1,n})
=a_{j_n,n}
=\tau_n'(w_r)
\geq a
\end{equation*}
so that also $z(\phi)\geq a$ for all $\phi\in(0,\phi_a)$.

\emph{Step\,3.}
For later use we show
\begin{equation*}
\lim_{\hat\phi\to0}\limsup_{n\to\infty}I_n(\hat\phi)=0
\qquad\text{ for }I_n(\hat\phi)=2\int_0^{\hat\phi}\sqrt{2z_n(\phi)\phi}\,\d\phi\,,
\end{equation*}
which due to $2\int_0^{\hat\phi}\sqrt{2z(\phi)\phi}\,\d\phi\leq\liminf_{n\to\infty}I_n(\hat\phi)$ by Fatou's lemma automatically also implies
\begin{equation*}
\lim_{\hat\phi\to0}I(\hat\phi)=0
\qquad\text{ for }I(\hat\phi)=2\int_0^{\hat\phi}\sqrt{2z(\phi)\phi}\,\d\phi\,.
\end{equation*}
The former limit makes use of the previous step and thus requires continuity of $\tau$ in $0$ (while the latter could also be obtained without).
For the proof, first note that if $\tau'(0)<\infty$,
then all $z_n$ are uniformly bounded above by $\tau'(0)$ so that the desired statement trivially holds.
Hence, in the following we assume $\tau'(0)=\infty$.
Now for arbitrary $r>0$ we will show existence of some $\phi_r>0$ and $N>0$ with $I_n(\phi_r)\leq r$ for all $n\geq N$,
which by the arbitrariness of $r$ and the monotonicity of $I_n$ implies the desired statement.
To this end pick $\varphi>0$ such that $z(\varphi)>r/\tau^{-1}(r)$ and $\phi_r>0$ such that $z(\phi_r)>2z(\varphi)$
(such $\varphi$ and $\phi_r$ exist by the previous step).
Furthermore, let $N,K>0$ such that
\begin{equation*}
z_n(\varphi)\geq K>r/\tau^{-1}(r)
\quad\text{and}\quad
z_n(\phi_r)\geq2z_n(\varphi)
\qquad\text{for all }n>N\,.
\end{equation*}
Letting $\hat w$ denote the (unique) point such that $\tau(\hat w)=K\hat w$, we now have $\hat w\leq\tau^{-1}(r)$ and thus
\[\tau(\hat w)\leq r\,.\]
Furthermore, $\tau(\hat w)=K\hat w\leq z_n(\varphi)\hat w$ implies for any $\delta>0$ the existence of a phase field function $\psi$ with
\begin{equation*}
\max\{\psi(y)\,|\,y\in\R\}\geq\phi_r
\quad\text{as well as}\quad
\int_\R\psi\,\d y=\hat w
\quad\text{and}\quad
F^{c_n}[\psi]\leq\tau_n(\hat w)+\delta
\end{equation*}
(where without loss of generality the maximum is achieved in $y=0$):
Indeed, a $\psi$ satisfying the latter two exists by definition of $\tau_n$, and if $\psi(y)<\phi_r$ for all $y\in\R$, then by the monotonicity of $z_n$ we have
\begin{equation*}
\tau_n(\hat w)\geq F^{c_n}[\psi]>z_n(\phi_r)\int_\R\psi\,\d y\geq2z_n(\varphi)\hat w\geq2\tau(\hat w)\geq2\tau_n(\hat w)\,,
\end{equation*}
which is a contradiction.
For such a $\psi$ \cref{thm:ModicaMortolaTrick} implies
\begin{equation*}
I_n(\phi_r)
\leq\int_{-\infty}^0\frac12|\psi'|^2+\psi z(\psi)\,\d y+\int_0^{\infty}\frac12|\psi'|^2+\psi z(\psi)\,\d y
=F^{c_n}[\psi]
\leq\tau_n(\hat w)+\delta
\leq\tau(\hat w)+\delta
\leq r+\delta\,,
\end{equation*}
where $\delta$ was arbitrary, thus $I_n(\phi_r)\leq r$ for all $n\geq N$.

\emph{Step\,4.}
Let $\tau^{z}$ be the cost induced by $c(\psi)=z(\psi)\psi$ via \eqref{eqn:inverseProblem}.
It remains to show $\tau^{z}\leq\tau$ and $\tau^{z}\geq\tau$.
As for the former, fix $\delta>0$ and $b\in(1,1+\delta]$.
By \cref{thm:rescaling}, for $n$ large enough we have
\begin{align*}
b^3z_n(\phi/b)&\geq z(\phi)\quad\text{for all }\phi\in[\delta,1/\delta)\,.
%b^3\tau_n(w)&\geq\tau(w)\quad\text{for all }w\in[\delta,\infty)\,.
\end{align*}
Note that
\begin{equation*}
\tau^{\tilde z_n}(w)=b^3\tau_n(w)
\quad\text{for}\quad
\tilde z_n(\phi)=b^3z_n(\phi/b)\,,\:
\tilde c_n(\phi)=\phi\tilde z_n(\phi)\,,
\end{equation*}
which can either be seen directly from \cref{thm:genUrbPln} or from the identity
\begin{equation*}
F^{\tilde c_n}[\varphi]
=\int_\R\frac12|\varphi'|^2+b^3z_n(\varphi/b)\varphi\,\d x
=b^3\int_\R\frac12|\psi'|^2+z_n(\psi)\psi\,\d x
=b^3F^{c_n}[\psi]
\end{equation*}
for all $\varphi:\R\to[0,\infty)$ and $\psi=\varphi(\cdot/b)/b$, which automatically have same mass $\int_\R\varphi\,\d y=\int_\R\psi\,\d y$.
Now
% let $[r,R)\supset(\delta,1/\delta)$ be the maximal half-open interval such that
% \begin{equation*}
% b^3z_n(\psi/b)\geq z(\psi)\quad\text{for all }\psi\in[r,R)
% \end{equation*}
% and
define
\begin{equation*}
\hat c(\phi)=\hat z(\phi)\phi
\quad\text{ for }\quad
\hat z(\phi)=\begin{cases}
\tilde z_n(\phi)&\text{if }\phi\in[\delta,1/\delta),\\
z(\phi)&\text{else.}
\end{cases}
\end{equation*}
By \cref{thm:monotonicityProperty} we have $\tau^{\hat z}\geq\tau^{z}$.
We aim to show
\begin{equation*}
\tau^{\hat z}(w)\leq\tau^{\tilde z_n}(w)+I(\delta)+\delta
\qquad\text{for all }w\geq0\text{ with }w\tau(w)\leq\frac1{8\delta^3}\,,
\end{equation*}
since this implies
\begin{equation*}
\tau^{z}(w)\leq\tau^{\hat z}(w)\leq b^3\tau_n(w)+I(\delta)+\delta\leq b^3\tau(w)+I(\delta)+\delta
\qquad\text{for all }w\geq0\text{ with }w\tau(w)\leq\frac1{8\delta^3}\,,
\end{equation*}
which by the arbitrariness of $\delta>0$ and $b\in(1,1+\delta]$ yields $\tau^{z}(w)\leq\tau(w)$ for all $w\geq0$.
To this end it suffices to construct a phase field function $\psi$ of mass $w$ (with $w\tau(w)\leq\frac1{8\delta^3}$) which satisfies $F^{\hat c}[\psi]\leq\tau^{\tilde z_n}(w)+I(\delta)+\delta$.
Since $\hat c$ is just a modification of $\tilde c_n$,
we start from a phase field function $\tilde\psi:\R\to[0,\infty)$ with mass $\int_\R\tilde\psi\,\d y=w$ and phase field energy $F^{\tilde c_n}[\tilde\psi]=\tau^{\tilde z_n}(w)+\delta/2$,
where by \cref{thm:phaseFieldProperties} we may assume $\tilde\psi$ to be even and decreasing on $[0,\infty)$.
Note that we may assume $\tilde\psi(0)<1/\delta$ since otherwise \cref{thm:energyEstimate} would imply $F^{\tilde c_n}[\tilde\psi]\geq\frac1{4w\delta^3}\geq2\tau(w)$,
which for $\delta$ small enough would be a contradiction.
Unfortunately, $F^{\hat c}[\tilde\psi]$ might be infinite, since $\hat c$ differs from $\tilde c_n$ for small values of $\tilde\psi$.
Thus we need to modify $\tilde\psi$.
Let $t=\min\{y\geq0\,|\,\tilde\psi(y)\leq\delta\}$ and choose $T>t$ and $\psi_\delta:[t,T]\to[0,\delta]$ monotonically decreasing such that
\begin{equation*}
\psi_\delta(t)=\delta
\quad\text{ and }\quad
\psi_\delta(T)=0
\quad\text{ as well as }\quad
\int_t^T\frac12|\psi_\delta'|^2+z(\psi_\delta)\psi_\delta\,\d x\leq
%\int_0^{\delta}\sqrt{2z(\psi)\psi}\,\d\psi+\delta/4=
I(\delta)/2+\delta/4
\end{equation*}
($T$ and $\psi_\delta$ exist by \cref{thm:ModicaMortolaTrick}).
We now assemble a new phase field function by
\begin{equation*}
\psi(y)=\begin{cases}
\tilde\psi(y)&\text{if }y\in(-t,t),\\
\psi_\delta(|y|)&\text{else.}
\end{cases}
\end{equation*}
If $\int_\R\psi\,\d x=w$, then indeed
\begin{multline*}
% \tau^{\hat z}(w)
% \leq\int_\R\frac12|\psi'|^2+\hat z(\psi)\psi\,\d y
F^{\hat c}[\psi]
=2\int_t^T\frac12|\psi_\delta'|^2+z(\psi_\delta)\psi_\delta\,\d y+\int_{-t}^t\frac12|\tilde\psi'|^2+\tilde c_n(\tilde\psi)\,\d y
\leq I(\delta)+\tfrac\delta2+F^{\tilde c_n}[\tilde\psi]-2\int_t^\infty\frac12|\tilde\psi'|^2+\tilde c_n(\tilde\psi)\,\d y\\
\leq I(\delta)+\delta+\tau^{\tilde z_n}(w)-2\int_t^\infty\frac12|\tilde\psi'|^2+\tilde c_n(\tilde\psi)\,\d y\,.
\end{multline*}
If $\int_\R\psi\,\d y>w$, we simply cut out a symmetric segment around $y=0$ to regain mass $w$, thereby reducing $F^{\hat c}[\psi]$ even further.
If on the other hand $\int_\R\psi\,\d y=w-\Delta w$ for some $\Delta w>0$,
we insert a segment of value $\tilde\psi(0)$ and width $\Delta w/\tilde\psi(0)$ at $y=0$ to regain mass $w$,
which inceases $F^{\hat c}[\psi]$ by
\begin{multline*}
\tilde c_n(\tilde\psi(0))\frac{\Delta w}{\tilde\psi(0)}
=\tilde z_n(\tilde\psi(0))\Delta w
<\tilde z_n(\tilde\psi(0))2\int_t^\infty\tilde\psi\,\d y\\
\leq2\int_t^\infty\tilde z_n(\tilde\psi)\tilde\psi\,\d y
\leq2\int_t^\infty\frac12|\tilde\psi'|^2+ \tilde c_n(\tilde\psi)\,\d y\,,
\end{multline*}
where we used $\Delta w=2\int_t^\infty\tilde\psi\,\d y-2\int_t^T\psi_\delta\,\d y<2\int_t^\infty\tilde\psi\,\d y$.
Summarizing, in all cases $F^{\hat c}[\psi]\leq\tau^{\tilde z_n}(w)+I(\delta)+\delta$, as desired.

\emph{Step\,5.}
The remaining inequality $\tau^{z}\geq\tau$ is shown analogously.
Fix $\delta>0$ and $b\in(1-\delta,1]$, then by \cref{thm:rescaling}, for all $n$ large enough we have
\begin{align*}
b^3z_n(\phi/b)&\leq z(\phi)+\delta
\qquad\text{for all }\phi\in[\delta,1/\delta)\,,\\
\tau_n(w)&\geq\tau(w)-\delta
\qquad\text{for all }w\in[\delta,\infty)\text{ with }w\tau(w)\leq1/(8\delta^3)\,.
\end{align*}
Note that the phase field cost $\phi\mapsto(z(\phi)+\delta)\phi$ induces the transport cost $w\mapsto\tau^{z}(w)+\delta w$.
Again abbreviate $\tilde z_n(\phi)=b^3z_n(\phi/b)$ and $\tilde c_n(\phi)=\tilde z_n(\phi)\phi$ and define
\begin{equation*}
\hat c(\phi)=\hat z(\phi)\phi
\quad\text{ for }\quad
\hat z(\phi)=\begin{cases}
\tilde z_n(\phi)&\text{if }\phi\in[\delta,1/\delta),\\
z(\phi)+\delta&\text{else,}
\end{cases}
\end{equation*}
then by \cref{thm:monotonicityProperty} we have $\tau^{\hat z}(w)\leq\tau^{z}(w)+\delta w$ for all $w\geq0$.
For arbitrary $w\geq\delta$ with $w\tau(w)\leq1/(8\delta^3)$ we now again seek some phase field function $\psi:\R\to[0,\infty)$ with mass $\int_\R\psi\,\d y=w$ and
\begin{equation*}
F^{\tilde c_n}[\psi]
%b^3\tau_n(w)
\leq\tau^{\hat z}(w)+b^{5/2}I_n(\delta/b)+\delta\,,
\end{equation*}
since this implies $b^3\tau_n(w)=\tau^{\tilde z_n}(w)\leq\tau^{\hat z}(w)+b^{5/2}I_n(\delta/b)+\delta$ and thus
\begin{equation*}
\tau^{z}(w)\geq\tau^{\hat z}(w)-\delta w\geq b^3\tau_n(w)-b^{5/2}I_n(\delta/b)-\delta-\delta w\geq b^3\tau(w)-b^3\delta-b^{5/2}I_n(\delta/b)-\delta-\delta w
\qquad\text{for all }w\geq\delta\text{ with }w\tau(w)\leq\frac1{8\delta^3}\,,
\end{equation*}
which by the arbitrariness of $\delta>0$ and step\,3 finally implies $\tau^{z}\geq\tau$.
We start from a phase field function $\hat\psi:\R\to[0,\infty)$ with mass $\int_\R\hat\psi\,\d y=w$ and phase field energy $F^{\hat c}[\hat\psi]=\tau^{\hat z}(w)+\delta/2$,
where again as in the previous step we may assume $\hat\psi$ to be even and decreasing on $[0,\infty)$ with $\hat\psi(0)<1/\delta$.
We now modify $\hat\psi$ where it takes values smaller than $\delta$ to account for the fact that $\hat c$ differs from $\tilde c_n$ in that region.
To this end let $t=\min\{y\geq0\,|\,\hat\psi(y)\leq\delta\}$ and choose $T>t$ and $\psi_\delta:[t,T]\to[0,\delta]$ monotonically decreasing such that
\begin{equation*}
\psi_\delta(t)=\delta
\quad\text{ and }\quad
\psi_\delta(T)=0
\quad\text{ as well as }\quad
\int_t^T\frac12|\psi_\delta'|^2+\tilde c_n(\psi_\delta)\,\d x\leq\int_0^{\delta}\sqrt{2\tilde c_n(\psi)}\,\d\psi+\delta/4=b^{5/2}I_n(\delta/b)/2+\delta/4\,.
\end{equation*}
We now assemble a new phase field function by
\begin{equation*}
\psi(y)=\begin{cases}
\hat\psi(y)&\text{if }y\in(-t,t),\\
\psi_\delta(|y|)&\text{else.}
\end{cases}
\end{equation*}
If $\int_\R\psi\,\d x=w$, then indeed
\begin{multline*}
% \tau^{\hat z}(w)
% \leq\int_\R\frac12|\psi'|^2+\hat z(\psi)\psi\,\d y
F^{\tilde c_n}[\psi]
=2\int_t^T\frac12|\psi_\delta'|^2+\tilde c_n(\psi_\delta)\,\d y+\int_{-t}^t\frac12|\hat\psi'|^2+\hat c(\hat\psi)\,\d y
\leq b^{5/2}I_n(\delta/b)+\tfrac\delta2+F^{\hat c}[\hat\psi]-2\int_t^\infty\frac12|\hat\psi'|^2+\hat c(\hat\psi)\,\d y\\
\leq b^{5/2}I_n(\delta/b)+\delta+\tau^{\hat z}(w)-2\int_t^\infty\frac12|\hat\psi'|^2+\hat c(\hat\psi)\,\d y\,.
\end{multline*}
If $\int_\R\psi\,\d y>w$, we again cut out a symmetric segment around $y=0$ to regain mass $w$, thereby reducing $F^{\tilde c_n}[\psi]$ even further.
If on the other hand $\int_\R\psi\,\d y=w-\Delta w$ for some $\Delta w>0$,
we insert a segment of value $\hat\psi(0)$ and width $\Delta w/\hat\psi(0)$ at $y=0$ to regain mass $w$,
which inceases $F^{\tilde c_n}[\psi]$ by
\begin{equation*}
\tilde c_n(\hat\psi(0))\frac{\Delta w}{\hat\psi(0)}
=\hat z(\hat\psi(0))\Delta w
<\hat z(\hat\psi(0))2\int_t^\infty\hat\psi\,\d y
\leq2\int_t^\infty\hat z(\hat\psi)\hat\psi\,\d y
\leq2\int_t^\infty\frac12|\hat\psi'|^2+ \hat c(\hat\psi)\,\d y\,,
\end{equation*}
where we used $\Delta w=2\int_t^\infty\hat\psi\,\d y-2\int_t^T\psi_\delta\,\d y<2\int_t^\infty\hat\psi\,\d y$.
Summarizing, in all cases $F^{\tilde c_n}[\psi]\leq\tau^{\hat z}(w)+b^{5/2}I_n(\delta/b)+\delta$, as desired.
\end{proof}

\begin{remark}[Nonuniqueness of phase field cost]
If $\tau$ is nondifferentiable in some point $w>0$ (that is, its superdifferential is set-valued) there can be multiple solutions to \cref{pbm:inverseProblem}.
An example for two phase field costs $c$ inducing the same nondifferentiable $\tau$ is given in \cref{exm:urbanPlanning,exm:urbanPlanningII}.
The role of differentiability will become clear in \cref{thm:necessaryCondition,thm:sufficientCondition},
where we will see that \eqref{eqn:convolutionProblem} for all $t\geq0$ is necessary and sufficient only for differentiable $\tau$.
%Indeed, we will see in \cref{thm:necessaryCondition,thm:sufficientCondition} that \eqref{eqn:convolutionProblem} for all $t\geq0$ is only a sufficient condition in this case (else it is also necessary by \cref{thm:necessaryCondition}).
\end{remark}

%\todo{Is the linear deconvolution problem uniquely solvable if it holds for all $t$? Then the above would be the only possibility for nonuniqueness.}

After having shown the existence of a phase field cost $c$ inducing a given admissible transport cost $\tau$, we gather some first a priori estimates on $c$ in the next two theorems.
Those will later be used in deriving the equivalent linear inverse problem \eqref{eqn:convolutionProblem}.

\begin{theorem}[Slopes of transport cost]\label{thm:aPrioriEstimateI}
Let $z(\phi)=c(\phi)/\phi$ for $\phi>0$ be nonincreasing. Denoting the right derivative by $(\cdot)'$ we have
\begin{equation*}
(\tau^z)'(0)=\lim_{\phi\searrow0}z(\phi)
\qquad\text{and}\qquad
\lim_{w\to\infty}(\tau^z)'(w)=\lim_{\phi\to\infty}z(\phi)\,.
\end{equation*}
\end{theorem}
\begin{proof}
We proceed in four steps.
In the following we simply write $\tau$ for $\tau^z$.

$\tau'(0)\leq\lim_{\psi\searrow0}z(\psi)$:
First note that the limit is well-defined (since $z$ is decreasing).
Assume the right-hand side to be finite (otherwise there is nothing to show).
Given $w,\delta>0$, we define $\psi_\delta^w(y)=\delta\max\{0,1-\delta|y|/w\}$ and obtain
\begin{equation*}
\tau(w)
\leq F^c[\psi_\delta^w]
=\delta^3/w+\int_\R z(\psi_\delta^w)\psi_\delta^w\,\d y
\leq\delta^3/w+w\lim_{\phi\searrow0}z(\phi)\,.
\end{equation*}
Letting $\delta\to0$ we obtain $\tau(w)\leq w\lim_{\phi\searrow0}z(\phi)$, which together with $\tau\geq0$ implies the desired result.

$\tau'(0)\geq\lim_{\phi\searrow0}z(\phi)$:
By \cref{thm:energyEstimate}, for any $\hat\phi>0$ we can find $\delta>0$ such that for all $w\leq\delta$ and all phase fields $\psi_w:\R\to[0,\infty)$ with mass $\int_\R\psi_w\,\d y=w$ and phase field energy $F^c[\psi_w]\leq2\tau(w)$
we have $\psi_w(y)\leq\hat\phi$ for almost all $y\in\R$.
Thus, those $\psi_w$ also satisfy $F^c[\psi_w]\geq z(\hat\phi)\int_\R\psi_w\,\d y=z(\hat\phi)w$ so that
\begin{equation*}
\tau(w)
=\inf\left\{F^c[\psi_w]\,\middle|\,\psi_w:\R\to[0,\infty),\,\int_\R\psi_w\,\d y=w,\,F^c[\psi_w]\leq2\tau(w)\right\}
\geq z(\hat\phi)w
\end{equation*}
for all $w\leq\delta$.
Consequently, $\tau'(0)=\lim_{w\to0}\tau(w)/w\geq z(\hat\phi)$.
The arbitrariness of $\hat\phi>0$ implies the desired inequality.

$\lim_{w\to\infty}\tau'(w)\geq\lim_{\phi\to\infty}z(\phi)$:
Abbreviate $a=\lim_{\phi\to\infty)}z(\phi)$ and define $\tilde z(\phi)=a$ as well as $\tilde c(\phi)=\tilde z(\phi)\phi$.
Using \cref{thm:genUrbPln,thm:monotonicityProperty}, we obtain $aw=\tau^{\tilde z}(w)\leq\tau(w)$ for all $w\geq0$,
which implies the desired inequality due to the concavity of $\tau$.

$\lim_{w\to\infty}\tau'(w)\leq\lim_{\phi\to\infty}z(\phi)$:
Assume to the contrary that there exists $\hat\phi$ with $z(\hat\phi)=b<\lim_{w\to\infty}\tau'(w)$.
Now let $\psi$ be a phase field with maximum value $\hat\phi$ at $y=0$, finite cost $F^c[\psi]>0$ and mass $\int_\R\psi\,\d y=W>0$.
By squeezing in a constant segment with value $\hat\phi$ and width $(w-W)/\hat\phi$ at $y=0$ we obtain a new phase field $\psi_w$ with mass $w=\int_\R\psi_w\,\d y>W$. Consequently,
\begin{equation*}
\tau(w)
\leq F^c[\psi_w]
=F^c[\psi]+(w-W)b
\quad\text{for all }w>W\,,
\end{equation*}
however, this contradicts $b<\lim_{w\to\infty}\tau'(w)$ due to the concavity of $\tau$.
\end{proof}

\begin{theorem}[A priori estimates on $c$]
Let $z(\phi)=c(\phi)/\phi$ for $\phi>0$ be nonincreasing, and let $c$ induce an admissible transport cost $\tau$ via \eqref{eqn:inverseProblem}. Then
\begin{equation*}
\int_0^1\sqrt{c(\phi)}\,\d\phi<\infty\,,
\qquad\text{wich implies}\qquad
\liminf_{\phi\searrow0}z(\phi)\phi^3=0\,.
\end{equation*}
\end{theorem}
\begin{proof}
For given phase field function $\psi:\R\to[0,\infty)$, let $\hat\phi=\max\{\psi(y)\,|\,y\in\R\}$ denote its maximum (which without loss of generality it achieves at $y=0$).
By \cref{thm:ModicaMortolaTrick},
\begin{equation*}
F^c[\psi]
=\int_{-\infty}^0\frac12|\psi'|^2+c(\psi)\,\d y+\int_0^{\infty}\frac12|\psi'|^2+c(\psi)\,\d y
%\geq\int_\R\sqrt{2c(\psi)}|\psi'|\,\d y
\geq2\int_0^{\hat\phi}\sqrt{2c(\phi)}\,\d\phi\,.
\end{equation*}
Thus, if the right-hand side were infinite for $\hat\phi=1$ (and thus for all $\hat\phi>0$ due to the monotonicity of $z$), then so would be $\tau$.
\end{proof}

\section{Obtainable transport costs $\tau$}\label{sec:obtainableCosts}
In this section we show that any transport cost $\tau:[0,\infty)\to[0,\infty)$ induced by a phase field cost $c:[0,\infty)\to[0,\infty)$ via \eqref{eqn:inverseProblem} is actually an admissible transport cost.
Combining this result with \cref{thm:existence}, any phase field cost $\tilde c:[0,\infty)\to[0,\infty)$ can be replaced by an equivalent phase field cost $c:[0,\infty)\to[0,\infty)$ such that $z(\psi)=c(\psi)/\psi$ is nonincreasing.

\begin{theorem}[Properties of $\tau$]\label{thm:tauProperties}
Let the phase field cost $c:[0,\infty)\to[0,\infty)$ be Borel measurable with $c(0)=0$, and let $c$ induce the transport cost $\tau:[0,\infty)\to[0,\infty)$ via \eqref{eqn:inverseProblem}. Then $\tau$ is admissible.
\end{theorem}
\begin{proof}
Since $F^c$ is nonnegative, we trivially have $\tau\geq0$ and $\tau(0)=F^c[0]=0$.

It is also straightforward to see that $\tau$ is nondecreasing.
Indeed, given any $\psi:\R\to[0,\infty)$ with mass $\int_\R\psi\,\d y=w$ and finite cost $F^c[\psi]$
(where by \cref{thm:phaseFieldProperties} we may assume $\psi$ to be bounded, even, and decreasing on $[0,\infty)$)
we can define $\psi_R(y)=\psi(|y|+R)$ for any $R\geq0$.
Since $R\mapsto\int_\R\psi_R\,\d y$ is Lipschitz continuous and nonincreasing with limit $\lim_{R\to\infty}\int_\R\psi_R\,\d y=0$,
for any $\hat w\in(0,w)$ there exists some $R(\hat w)>0$ with $\int_\R\psi_{R(\hat w)}\,\d y=\hat w$ and obviously
\begin{equation*}
\tau(\hat w)
\leq F^c[\psi_{R(\hat w)}]
=\int_{\R\setminus[-R(\hat w),R(\hat w)]}\frac12|\psi'|^2+c(\psi)\,\d y
\leq F^c[\psi]\,.
\end{equation*}
Taking the infimum over all phase field functions $\psi$ with mass $w$ we thus obtain $\tau(\hat w)\leq\tau(w)$.

Similarly, we can show continuity of $\tau$ in $0$.
Indeed, fix some $\psi:\R\to[0,\infty)$ (even and decreasing on $[0,\infty)$) with finite cost $F^c[\psi]$ and abbreviate
\begin{equation*}
S(R)=F^c[\psi_R]=\int_{\R\setminus[-R,R]}\rho\,\d y
\end{equation*}
for the Lebesgue integrable $\rho(y)=\frac12|\psi'(y)|^2+c(\psi(y))$.
The function $S:[0,\infty)\to[0,F^c[\psi]]$ is absolutely continuous and nonincreasing with $\lim_{R\to\infty}S(R)=0$.
Thus, for any $\delta>0$ we can find $R>0$ with $0<S(R)\leq\delta$.
Abbreviating $\hat w=\int_\R\psi_R\,\d y>0$ we thus obtain
\begin{equation*}
\tau(w)
\leq\tau(\hat w)
\leq F^c[\psi_R]
\leq\delta
\end{equation*}
for all $w\leq\hat w$, implying the desired continuity.

Finally, $\tau$ is concave.
Indeed, fix some arbitrary $w>0$.
For given small $\varepsilon,\delta>0$ we can find a phase field function $\psi:\R\to[0,\infty)$ with mass $\int_\R\psi\,\d y=w$ and phase field energy $F^c[\psi]\leq\tau(w)+\varepsilon$
as well as some $\hat y\in\R$ with $z(\psi(\hat y))=c(\psi(\hat y))/\psi(\hat y)\leq a(\varepsilon)+\delta$, where we abbreviated $a(\varepsilon)=\inf\{z(\psi(y))\,|\,y\in\R\}$.
Again, by \cref{thm:phaseFieldProperties} we may assume $\psi$ to be even.
Now let $\hat w\neq w$.
If $\hat w>w$, then by defining
\begin{equation*}
\psi_{\hat w}(y)
=\begin{cases}
\psi(y)&\text{if }y\leq\hat y\,,\\
\psi(\hat y)&\text{if }\hat y\leq y\leq\hat y+\frac{\hat w-w}{\psi(\hat y)}\,,\\
\psi(y-\frac{\hat w-w}{\psi(\hat y)})&\text{else},
\end{cases}
\end{equation*}
which satisfies $\int_\R\psi_{\hat w}\,\d y=\hat w$, we derive
\begin{equation*}
\tau(\hat w)
\leq F^c[\psi_{\hat w}]
=F^c[\psi]+z(\psi(\hat y))(\hat w-w)
\leq\tau(w)+\varepsilon+(a(\varepsilon)+\delta)(\hat w-w)\,.
\end{equation*}
Similarly, if $\hat w<w$, then there exists $R(\hat w)>0$ such that $\int_\R\psi_{R(\hat w)}\,\d y=\hat w$ so that
\begin{multline*}
\tau(\hat w)
\leq F^c[\psi_{R(\hat w)}]
=\int_\R\frac12|\psi'|^2+z(\psi)\psi\,\d y-\int_{-R(\hat w)}^{R(\hat w)}\frac12|\psi'|^2+z(\psi)\psi\,\d y\\
\leq F^c[\psi]-a(\varepsilon)\int_{-R(\hat w)}^{R(\hat w)}\psi\,\d y
%\leq F^c[\psi]+a(\varepsilon)(\hat w-w)
\leq\tau(w)+\varepsilon+a(\varepsilon)(\hat w-w)\,.
\end{multline*}
Summarizing and using the arbitrariness of $\delta$ we obtain
\begin{equation*}
\tau(\hat w)
\leq\tau(w)+\varepsilon+a(\varepsilon)(\hat w-w)
\end{equation*}
for all $\hat w\geq0$.
Now the parameters $a(\varepsilon)$ are uniformly bounded below by $0$ and above by $\frac{\tau(w)+1}{w}$ (otherwise the previous inequality would be violated for $\hat w=0$),
hence we can consider a sequence $\varepsilon_n\to0$ as $n\to\infty$ such that $a(\varepsilon_n)\to a\in[0,\infty)$ as $n\to\infty$.
Taking this limit in the above inequality we arrive at
\begin{equation*}
\tau(\hat w)
\leq\tau(w)+a(\hat w-w)
\end{equation*}
for all $\hat w\geq0$, where $a\in[0,\infty)$ of course depends on $w$.
Concavity now follows from the arbitrariness of $w$.
\end{proof}

\section{Existence and properties of phase field profile}\label{sec:existencePhaseField}
The derivation of the equivalent linear convolution problem \eqref{eqn:convolutionProblem} will make use of properties of optimal phase field functions.
To this end we now provide the existence of minimizers of the phase field energy $F^c$ under a mass constraint.
The essential issue here is that the phase field mass $\psi\mapsto\int_\R\psi\,\d y$ is not weakly continuous on the Hilbert space $H^1(\R)$ due to the unboundedness of the domain.
As a consequence, existence of minimizers only holds for certain masses.

\begin{theorem}[Existence of optimal phase field function]\label{thm:existencePhaseField}
Let the mass-specific phase field cost $z$ be nonincreasing and lower semi-continuous with $\int_0^1\sqrt{c(\phi)}\,\d\phi<\infty$,
and let $w>0$ such that $\tau^z(w)\neq(\tau^z)'(0)w$, then the optimization problem
\begin{equation*}
\tau^z(w)=\min\left\{F^c[\psi]\,\middle|\,\psi:\R\to[0,\infty),\,\int_\R\psi\,\d y=w\right\}
\end{equation*}
has a minimizer $\psi\in H^1(\R)$ which is bounded, continuous, even, and decreasing on $[0,\infty)$.
\end{theorem}
\begin{proof}
We will simply write $\tau$ for $\tau^z$ and $c(\phi)=z(\phi)\phi$.
First note that $F^c$ is bounded below by $0$, and there exists some $\psi:\R\to[0,\infty)$ with $\int_\R\psi\,\d y=w$ and $F^c[\psi]<\infty$.
Now consider a minimizing sequence $\psi_n:\R\to[0,\infty)$, $n=1,2,\ldots$, with $\int_\R\psi_n\,\d y=w$ and $F^c[\psi_n]\to\tau(w)$ monotonically as $n\to\infty$.
By \cref{thm:phaseFieldProperties} we may assume the $\psi_n$ to be even and decreasing on $[0,\infty)$.
We now proceed in steps.

\emph{Step\,1.}
The sequence $\psi_n$ is uniformly bounded in $H^1(\R)$.
Indeed, the $H^1$-seminorm $\int_\R|\psi_n'|^2\,\d y$ is uniformly bounded by $2F^c[\psi_1]$, so it remains to show boundedness of the $L^2$-norm.
To this end we note that $\{y\in\R\,|\,\psi_n(y)>1\}\subset[-\frac1{2w},\frac1{2w}]$ and estimate
\begin{multline*}
\int_\R\psi_n^2\,\d y
\leq\int_{\{y\in\R\,|\,\psi_n(y)>1\}}\psi_n^2\,\d y+\int_{\{y\in\R\,|\,\psi_n(y)\leq1\}}\psi_n\,\d y
\leq\int_{-\frac1{2w}}^{\frac1{2w}}\psi_n^2\,\d y+w\\
=\int_{-\frac1{2w}}^{\frac1{2w}}(\psi_n-\bar\psi_n)^2\,\d y+\int_{-\frac1{2w}}^{\frac1{2w}}2(\psi_n-\bar\psi_n)\bar\psi_n\,\d y+\int_{-\frac1{2w}}^{\frac1{2w}}\bar\psi_n^2\,\d y+w
=\int_{-\frac1{2w}}^{\frac1{2w}}(\psi_n-\bar\psi_n)^2\,\d y+\frac{\bar\psi_n^2}w+w\,,
\end{multline*}
where $\bar\psi_n=w\int_{-1/{2w}}^{1/{2w}}\psi_n\,\d y\leq w^2$ is the average value of $\psi_n$ on $[-\frac1{2w},\frac1{2w}]$.
With Poincar\'e's inequality on $[-\frac1{2w},\frac1{2w}]$ we thus obtain
\begin{equation*}
\int_\R\psi_n^2\,\d y
\leq K\int_{-\frac1{2w}}^{\frac1{2w}}|\psi_n'|^2\,\d y+w^3+w
\leq 2KF^c[\psi_1]+w^3+w
\end{equation*}
for $K$ the Poincar\'e constant, which is independent of $n$.

\emph{Step\,2.}
Since $\psi_n$ is uniformly bounded in $H^1(\R)$ there exists a weakly (and pointwise almost everywhere) converging subsequence (again indexed by $n$) with $\psi_n\rightharpoonup\psi$,
where $\psi\in H^1(\R)$ is nonnegative, even, and monotonically decreasing on $[0,\infty)$ (by Sobolev embedding it is also bounded and continuous).
Now Fatou's lemma and the lower semi-continuity of $c$ imply
\begin{equation*}
\int_\R c(\psi)\,\d y
\leq\int_\R\liminf_{n\to\infty}c(\psi_n)\,\d y
\leq\liminf_{n\to\infty}\int_\R c(\psi_n)\,\d y\,.
\end{equation*}
Together with the weak sequential lower semi-continuity of the $H^1$-seminorm this implies $F^c[\psi]\leq\liminf_{n\to\infty}F^c[\psi_n]=\tau(w)$.

\emph{Step\,3.}
For $\psi$ to be the desired minimizer, it remains to show $\int_\R\psi\,\d y=w$.
% To this end let $\hat\psi>0$ be such that $z(\hat\psi)>\tau(w)/w$ (which exists due to $\lim_{\psi\to0}z(\psi)=\tau'(0)>\tau(w)/w$ by \cref{thm:aPrioriEstimateI}).
% For $n$ large enough we have $\psi_n(0)>\hat\psi$, since otherwise $F^c[\psi_n]\geq z(\hat\psi)\int_\R\psi_n\,\d y=z(\hat\psi)w>\tau(w)$ by the monotonicity of $z$.
To this end we first show that for any $\delta>0$ we can find $L>0$ such that $\int_{\R\setminus[-L,L]}\psi_n\,\d y\leq\delta$ for all $n$.
Assume the contrary, that is, there exist sequences $n_i$ and $L_i$, $i=1,2,\ldots$, with $n_i,L_i\to\infty$ such that $\int_{\R\setminus[-L_i,L_i]}\psi_{n_i}\,\d y=\delta$.
Due to
\begin{equation*}
w-\delta=\int_{-L_i}^{L_i}\psi_{n_i}\,\d y\geq2L_i\psi_{n_i}(L_i)
\end{equation*}
this necessarily implies $\psi_{n_i}(L_i)\to0$ as $i\to\infty$.
Now define
\begin{equation*}
\tilde\psi_i(y)=\begin{cases}
\psi_{n_i}(y)&\text{if }y\in[-L_i,L_i]\,,\\
\psi_1(|y|-R_i)&\text{else},
\end{cases}
\end{equation*}
where $R_i$ is chosen such that $\psi_1(R_i)=\psi_{n_i}(L_i)$.
Obviously, $\int_{\R\setminus[-R_i,R_i]}\frac12|\psi_1'|^2+c(\psi_1)\,\d y\to0$ as $i\to\infty$ and $\int_\R\tilde\psi_i\,\d y\geq w-\delta$. Now
\begin{align*}
F^c[\psi_{n_i}]
&=F^c[\tilde\psi_i]-\int_{\R\setminus[-R_i,R_i]}\frac12|\psi_1'|^2+c(\psi_1)\,\d y+\int_{\R\setminus[-L_i,L_i]}\frac12|\psi_{n_i}'|^2+c(\psi_{n_i})\,\d y\\
&\geq\tau(w-\delta)-\int_{\R\setminus[-R_i,R_i]}\frac12|\psi_1'|^2+c(\psi_1)\,\d y+z(\psi_{n_i}(L_i))\int_{\R\setminus[-L_i,L_i]}\psi_{n_i}\,\d y\\
&>\tau(w-\delta)-\int_{\R\setminus[-R_i,R_i]}\frac12|\psi_1'|^2+c(\psi_1)\,\d y+z(\psi_{n_i}(L_i))\delta\,.
\end{align*}
Taking the limit $i\to\infty$, we obtain $z(\psi_{n_i}(L_i))\to\tau'(0)$ by \cref{thm:aPrioriEstimateI} and thus
\begin{equation*}
\tau(w)\geq\tau(w-\delta)+\tau'(0)\delta\,.
\end{equation*}
Together with the concavity property $\tau(w)\leq\tau(w-\delta)+\tau'(w-\delta)\delta\leq\tau(w-\delta)+\tau'(0)\delta$ this implies $\tau'(w-\delta)=\tau'(0)$ and thus $\tau(w-\delta)=\tau'(0)(w-\delta)$ as well as
\begin{equation*}
\tau(w)=\tau(w-\delta)+\tau'(0)\delta=\tau'(0)w\,,
\end{equation*}
the desired contradiction.

By the weak convergence $\psi_n\rightharpoonup\psi$ in $H^1((-L,L))$ we have
\begin{equation*}
\int_\R\psi\,\d y
\geq\int_{-L}^L\psi\,\d y
=\lim_{n\to\infty}\int_{-L}^L\psi_{n}\,\d y
\geq w-\delta
\end{equation*}
as well as
\begin{equation*}
\int_\R\psi\,\d y
=\int_\R\liminf_{n\to\infty}\psi_n\,\d y
\leq\liminf_{n\to\infty}\int_\R\psi_{n}\,\d y
=w
\end{equation*}
by Fatou's lemma.
The arbitrariness of $\delta>0$ concludes the proof.
\end{proof}

Just for the sake of completeness we briefly show nonexistence of optimal phase field functions in the region where $\tau$ is linear.
As mentioned before, the major reason is the lack of weak continuity of the mass constraint.
Indeed, it is straightforward to see that a minimizing sequence for mass $w$ would be given by $\psi_n(y)=\max\{0,\frac wn(1-\frac{|y|}n)\}$, $n=1,2,\ldots$, which converges weakly to $0$.

\begin{theorem}[Nonexistence of a minimizer for linear $\tau$]
Under the conditions of the previous theorem, assume $\tau^z(w)=(\tau^z)'(0)w$ for all $w\in[0,W]$.
Then for $w\in(0,W)$ there is no minimizer in \eqref{eqn:inverseProblem}.
\end{theorem}
\begin{proof}
Again we write $\tau$ for $\tau^z$.
Assume the converse, that is, there is some $\psi:\R\to[0,\infty)$ with $\int_\R\psi\,\d y=w$ and $F^c[\psi]=\tau(w)$.
By \cref{thm:phaseFieldProperties} we may assume $\psi$ to be even and to take its maximum in $y=0$.
In particular, $\psi(0)>0$.
By the monotonicity of $z$ and \cref{thm:aPrioriEstimateI} we have $z(\psi(0))\leq\tau'(0)$.
Now there are two cases.
If $z(\phi)=\tau'(0)$ for all $\phi\in[0,\psi(0)]$, then $F^c[\psi]=\int_\R\frac12|\psi'|^2\,\d y+\tau'(0)w>\tau(w)$, yielding a contradiction.
If on the other hand $z(\psi(0))<\tau'(0)$, then $\tilde\psi(y)=\psi(\max\{0,|y|-\frac{W-w}{2\psi(0)}\})$ satisfies
\begin{equation*}
\int_\R\tilde\psi\,\d y=W
\qquad\text{and}\qquad
\tau(W)
\leq F^c[\tilde\psi]
=F^c[\psi]+z(\psi(0))(W-w)
=\tau'(0)w+z(\psi(0))(W-w)
<\tau'(0)W
=\tau(W)\,,
\end{equation*}
yielding again a contradiction.
\end{proof}

\section{Identification of phase field cost via a deconvolution problem}\label{sec:deconvolution}
In this section we rigorously derive the linear deconvolution problem \eqref{eqn:convolutionProblem}.
We start by characterizing the minimizers from \cref{thm:existencePhaseField} as being minimizers of an alternative variational problem.

\begin{lemma}[Alternative characterization of optimal phase field function]\label{thm:alternativeMinimization}
Let $\tau=\tau^z$ for a nonincreasing, lower semi-continuous mass-specific phase field cost $z$.
Furthermore, let $w>0$ such that $\tau(w)\neq\tau'(0)w$, and let $\hat\psi:\R\to[0,\infty)$ be an optimal phase field function from \cref{thm:existencePhaseField}. Then
\begin{enumerate}
\item\label{enm:alternativeMinimization} for $\lambda=z(\hat\psi(0))$ we have
\begin{equation*}
\hat\psi\in\argmin\left\{F^c[\psi]-\lambda\int_\R\psi\,\d y\,\middle|\,\psi\in H^1(\R)\text{ nonnegative, even, and nonincreasing on }[0,\infty),\,\psi(0)=\hat\psi(0)\right\}\,,
\end{equation*}
\item $\lambda\geq\tau'(w)$ for the right derivative $\tau'(w)$ of $\tau$ in $w$,
\item $\lambda\leq\tau^{l}(w)$ for the left derivative $\tau^{l}(w)$ of $\tau$ in $w$.
\end{enumerate}
\end{lemma}
\begin{proof}
\begin{enumerate}
\item
Abbreviate $c(\psi)=z(\psi)\psi$ as well as $E[\psi]=F^c[\psi]-\lambda\int_\R\psi\,\d y$ and note $E[\hat\psi]=\tau(w)-\lambda w$.
Assume there were some feasible function $\psi$ with strictly smaller $E[\psi]$, and abbreviate $W=\int_\R\psi\,\d y$.
We can distinguish two cases.
If $W\leq w$, then $\psi_W(y)=\psi(\max\{0,|y|-\frac{w-W}{2\hat\psi(0)}\})$ satisfies
\begin{equation*}
\int_\R\psi_W\,\d y=w
\qquad\text{and}\qquad
\tau(w)
\leq F^c[\psi_W]
=E[\psi_W]+\lambda w
=E[\psi]+\lambda w
<E[\hat\psi]+\lambda w
=\tau(w)\,,
\end{equation*}
a contradiction.
If on the other hand $W>w$ then there exists $R>0$ so that $\psi_R(y)=\psi(|y|+R)$ satisfies
\begin{equation*}
\int_\R\psi_R\,\d y=w\,,
\qquad\text{and then}\qquad
\tau(w)
\leq F^c[\psi_R]
=E[\psi_R]+\lambda w
\leq E[\psi]+\lambda w
<E[\hat\psi]+\lambda w
=\tau(w)
\end{equation*}
(where the second inequality follows from $c(\psi(y))-\lambda\psi(y)\geq(z(\hat\psi(0))-\lambda)\psi(y)\geq0$ for all $y\in\R$),
which is again a contradiction.

\item
Assume the converse, then there is some $W>w$ with $\tau(W)>\tau(w)+\lambda(W-w)$.
Now $\psi_W(y)=\hat\psi(\max\{0,|y|-\frac{W-w}{2\hat\psi(0)}\})$ satisfies
\begin{equation*}
\int_\R\psi_W\,\d y=W
\qquad\text{and}\qquad
\tau(W)
\leq F^c[\psi_W]
=E[\psi_W]+\lambda W
=E[\hat\psi]+\lambda W
=\tau(w)+\lambda(W-w)\,,
\end{equation*}
a contradiction.

\item
Assume the converse, then there is some $W<w$ with $\tau(W)>\tau(w)+\lambda(W-w)$.
Now there exists $R>0$ so that $\psi_R(y)=\hat\psi(|y|+R)$ satisfies
\begin{equation*}
\int_\R\psi_R\,\d y=W\,,
\qquad\text{and then}\qquad
\tau(W)
\leq F^c[\psi_R]
=E[\psi_R]+\lambda W
\leq E[\hat\psi]+\lambda W
=\tau(w)+\lambda(W-w)\,,
\end{equation*}
a contradiction.
\qedhere
\end{enumerate}
\end{proof}

\begin{remark}[On the Lagrange multiplier]
The optimal phase field functions constructed in \cref{thm:genUrbPln} suggest that at a point $w$ with $\tau'(w)<\tau^l(w)$ there are two minimizers $\hat\psi_1$ and $\hat\psi_2$ of \eqref{eqn:inverseProblem}.
One of them is such that the Lagrange multiplier $\lambda$ from the previous result satisfies $\lambda=z(\hat\psi_1(0))=\tau'(w)$, while for the other one we have $\lambda=z(\hat\psi_2(0))=\tau^l(w)$.
It is not expected that the case $\lambda\in(\tau'(w),\tau^l(w))$ can occur.
\end{remark}

% Since $c$ up to a point only depends on $\tau$ up to $w$, one could continue $\tau$ beyond $w$ with slope $\tau^{'l}(w)$ which would implie $\lambda=\tau^{'l}(w)$.
% Probably there is one minimizer with $\lambda=\tau^{'l}$ and one with $\lambda=\tau^{'r}$ (but typically none in between, e.g. consider urban planning).
% This can be proved by showing $\Gamma(\text{weak }H^1)-\lim_{i\to\infty}E^{\lambda_i,z_i}=E^{\lambda,z}$ for $\lambda_i,z_i\to\lambda,z$ with $E^{\lambda,z}(\phi)=\int_\R\frac12|\phi'|^2+c(\phi)-\lambda\phi\,\d y+\lambda w+\iota_{\max\phi=z}(\phi)$.

In the following we will make use of the concept of a generalized inverse, as for instance also considered in \cite{EmHo13}
(while we consider nonincreasing functions on $[0,\infty)$, most other works such as \cite{EmHo13} consider nondecreasing functions on $\R$, but there is no essential difference in the resulting structure and properties).

\begin{theorem}[Generalized inverse]\label{thm:generalizedInverse}
For a nonincreasing lower semi-continuous function $z:[0,\infty)\to[0,\infty]$
its generalized inverse $z^{-1}:[0,\infty)\to[0,\infty]$ is defined as
\begin{equation*}
z^{-1}(x)=\inf\{y\in[0,\infty)\,|\,z(y)\leq x\}\,,
\end{equation*}
where by convention $\inf\emptyset=\infty$.
We have the following properties,
\begin{enumerate}
\item $z^{-1}$ is nonincreasing and lower semi-continuous,
\item $(z^{-1})^{-1}=z$,
\item\label{enm:integralFormula} $\int_0^\infty z\,\d y=\int_0^\infty z^{-1}\,\d x$.
\end{enumerate}
\end{theorem}
\begin{proof}
\begin{enumerate}
\item
Let $x_1<x_2$, then $z^{-1}(x_2)$ is the infimum over a larger set than is $z^{-1}(x_1)$, thus $z^{-1}(x_2)\leq z^{-1}(x_1)$ and $z^{-1}$ is nonincreasing.
To show the lower semi-continuity it suffices to consider sequences $x_n$, $n=1,2,\ldots$, with $x_n\to x$ from the right (size $z^{-1}$ is automatically lower semi-continuous from the left as a nonincreasing function).
By the lower semi-continuity of $z$ we have $z(z^{-1}(x_n))\leq x_n$ for all $n$.
Furthermore, the lower semi-continuity of $z$ implies $z(\lim_{n\to\infty}z^{-1}(x_n))\leq\lim_{n\to\infty}z(z^{-1}(x_n))\leq x$,
thus $\lim_{n\to\infty}z^{-1}(x_n)\geq z^{-1}(x)$.
\item
We have $(z^{-1})^{-1}(y)\leq z(y)$ by definition of the generalized inverse.
For the converse inequality, assume there exists some $y\in[0,\infty)$ with $(z^{-1})^{-1}(y)<z(y)$,
then this would imply the existence of some $x<z(y)$ with $z^{-1}(x)\leq y$.
However, the latter implies $z(y)\leq x$, a contradiction.
\item
Letting $\mathbf1_A$ denote the characteristic function of a set $A\subset\R^2$, by Fubini's theorem we have
\begin{multline*}
\int_0^\infty z(y)\,\d y
=\int_0^\infty\int_0^\infty\mathbf1_{\{x\in[0,\infty)\,|\,x<z(y)\}}(x)\,\d x\,\d y\\
=\int_0^\infty\int_0^\infty\mathbf1_{\{y\in[0,\infty)\,|\,y<z^{-1}(x)\}}(y)\,\d x\,\d y
=\int_0^\infty z^{-1}(x)\,\d x\,,
\end{multline*}
where we exploited the equivalence of $x<z(y)$ with $y<z^{-1}(x)$.
\qedhere
\end{enumerate}
\end{proof}

\begin{lemma}[Convolution identity]\label{thm:convolutionIdentity}
Let the two nonincreasing, lower semi-continuous functions $z,g:[0,\infty)\to[0,\infty]$ be related by
\begin{equation*}
g=(z^{-1}(\cdot))^{3/2}
\qquad\text{or equivalently}\qquad
z=g^{-1}((\cdot)^{2/3})\,,
\end{equation*}
(in the sense of generalized inverses) and define
\begin{equation*}
r(s)=\begin{cases}\frac{2\sqrt2}{3\sqrt{-s}}&\text{if }s<0,\\0&\text{else.}\end{cases}
\end{equation*}
Then, for any $t,\hat\phi\in[0,\infty)$ with $z(\hat\phi)=t$ or $\hat\phi=z^{-1}(t)$ we have
\begin{equation*}
\inf\left\{\int_\R\frac12|\psi'|^2+\max\{0,z(\psi)-t\}\psi\,\d y\,\middle|\,\psi\in H^1(\R)\text{ nonnegative, even, and nonincreasing on }[0,\infty),\,\psi(0)=\hat\phi\right\}
=\left[g*r\right](t)\,.
\end{equation*}
\end{lemma}
\begin{proof}
We have
\begin{align*}
&\inf\left\{\int_\R\frac12|\psi'|^2+\max\{0,z(\psi)-t\}\psi\,\d y\,\middle|\,\psi\in H^1(\R)\text{ nonnegative, even, and nonincreasing on }[0,\infty),\,\psi(0)=\hat\phi\right\}\\
%&=2\inf\left\{\int_0^\infty\frac12|\psi'|^2+(z(\psi)-t)\psi\,\d y\,\middle|\,\psi\in H^1([0,\infty))\text{ nonnegative and nonincreasing},\,\psi(0)=\hat\phi\right\}\\
&=2\inf\left\{\int_0^\infty\frac12|\psi'|^2+\max\{0,z(\psi)-t\}\psi\,\d y\,\middle|\,\psi\in H^1([0,\infty))\text{ nonnegative and nonincreasing},\,\psi(0)=\hat\phi\right\}\\
&=2\inf\left\{\int_0^\infty\frac12|\psi'|^2+\max\{0,z(\psi)-t\}\psi\,\d y\,\middle|\,\psi\in H^1([0,\infty))\text{ nonnegative},\,\psi(0)=\hat\phi\right\}\\
&=2\int_0^{\hat\phi}\sqrt{2\max\{0,z(\phi)-t\}\phi}\,\d\phi\\
&=2\int_0^\infty\sqrt{2\max\{0,z(\phi)-t\}\phi}\,\d\phi\,,
\end{align*}
where in the second equality we used the P\'olya--Szeg\"o inequality, in the third equality we used \cref{thm:ModicaMortolaTrick},
and in the last equality we used that $z$ is nonincreasing.
Now abbreviate
\begin{equation*}
h:[0,\infty)\to[0,\infty],\;
h(q)=\frac23g\left(t+\frac{q^2}2\right)\,.
\end{equation*}
This function is nonincreasing and lower semi-continuous, hence admits a generalized inverse, given by
\begin{equation*}
h^{-1}(p)=\sqrt{2\max\left\{0,g^{-1}(\tfrac32p)-t\right\}}\,.
\end{equation*}
Using the change of variables $\phi=(3p/2)^{2/3}$ and \cref{thm:generalizedInverse}\eqref{enm:integralFormula} we now obtain
\begin{multline*}
2\int_0^\infty\sqrt{2\max\{0,z(\phi)-t\}\phi}\,\d\phi
=2\int_0^\infty\sqrt{2\max\{0,z((3p/2)^{2/3})-t\}}\,\d p
=2\int_0^\infty h^{-1}(p)\,\d p\\
=2\int_0^\infty h(q)\,\d q
=\frac43\int_0^\infty g\left(t+\frac{q^2}2\right)\,\d q
=\frac43\int_{t}^\infty\frac{g(s)}{\sqrt{2(s-t)}}\,\d s
=\left[g*r\right](t)\,,
\end{multline*}
where in the second last step we used the change of variables $s=t+\frac{q^2}2$.
\end{proof}

\begin{theorem}[Necessary condition]\label{thm:necessaryCondition}
Let $\tau=\tau^z$ for a nonincreasing, lower semi-continuous mass-specific phase field cost $z$.
\Cref{eqn:convolutionProblem} holds for all $t\in\{\tau'(w)\,|\,w\geq0,\,\tau\text{ differentiable in }w\}$.
\end{theorem}
\begin{proof}
Let $\tau$ be differentiable at $w$ with $\tau'(w)=t$, where initially we assume $t\neq\tau'(0)$.
\Cref{thm:alternativeMinimization} now implies $t=z(\hat\psi(0))$, where $\hat\psi$ is an optimal phase field function of mass $\int_\R\hat\psi\,\d y=w$, as well as (abbreviating $c(\phi)=z(\phi)\phi$)
\begin{align*}
\tau(w)-\tau'(w)w
&=F^c[\hat\psi]-t\int_\R\hat\psi\,\d y\\
&=\min\left\{F^c[\psi]-t\int_\R\psi\,\d y\,\middle|\,\psi\in H^1(\R)\text{ nonnegative, even, and nondecreasing on }[0,\infty),\,\psi(0)=\hat\psi(0)\right\}\\
&=\min\left\{\int_\R\frac12|\psi'|^2+\max\{0,z(\psi)-t\}\psi\,\d y\,\middle|\,\right.\\
&\hspace{20ex}\left.\vphantom{\int_\R}\psi\in H^1(\R)\text{ nonnegative, even, and nondecreasing on }[0,\infty),\,\psi(0)=\hat\psi(0)\right\}\,,
\end{align*}
where in the last step we used that $z$ is nonincreasing.
By \cref{thm:convolutionIdentity}, the right-hand side equals $[g*r](t)$, while the left-hand side is nothing else than $[-\tau(-\cdot)]^\ast(t)$.

The case $t=\tau'(0)$ is treated separately:
By \cref{thm:aPrioriEstimateI} we have $t=\lim_{\phi\searrow0}z(\phi)$ so that $g=0$ on $[t,\infty)$.
Furthermore, it is a straightforward consequence of the Legendre--Fenchel conjugate and the admissibility of $\tau$ that $[-\tau(-\cdot)]^\ast(t)=0$,
hence indeed $[-\tau(-\cdot)]^\ast(t)=\left[g*r\right](t)$, as desired.
\end{proof}

\begin{remark}[Necessary condition is not sufficient]
The condition from \cref{thm:necessaryCondition} is necessary, but not sufficient.
For instance, consider the transport cost $\tau(w)=\min\{aw,bw+d\}$
and the cost $\tilde\tau$, in which the nondifferentiability of $\tau$ is smoothed out in a small neighbourhood of $d/(a-b)$ but which otherwise coincides with $\tau$.
Then $[-\tau(-\cdot)]^\ast$ coincides with $[-\tilde\tau(-\cdot)]^\ast$ on $\{a,b\}$, thus the phase field cost $\tilde c$ inducing $\tilde\tau$ satisfies the necessary conditions for inducing $\tau$.
\end{remark}

\begin{theorem}[Sufficient condition]\label{thm:sufficientCondition}
Let the lower semi-continuous, nonincreasing function $g$ solve \eqref{eqn:convolutionProblem} for all $t\geq0$ and an admissible transport cost $\tau$, then $\tau=\tau^z$ for $z=g^{-1}((\cdot)^{2/3})$.
\end{theorem}
\begin{proof}
The function $z$ is well-defined via \cref{thm:generalizedInverse}, and we set $c(\phi)=z(\phi)\phi$.

\emph{Step\,1.}
Let us abbreviate $E^t[\psi]=\int_\R\frac12|\psi'|^2+\max\{0,z(\psi)-t\}\psi\,\d y$ and $Z=\{t\geq0\,|\,z^{-1}(t)<\infty\}$,
then $Z=[T,\infty)$ or $Z=(T,\infty)$ for some $T\geq0$.
By \cref{thm:convolutionIdentity}, for all $t\in Z$ we have
\begin{align*}
[-\tau(-\cdot)]^\ast(t)
&=\left[g*r\right](t)\\
&=\inf\left\{E^t[\psi]\,\middle|\,\psi\in H^1(\R)\text{ nonnegative, even, and nonincreasing on }[0,\infty),\,\psi(0)=z^{-1}(t)\right\}\,.
\end{align*}
On the other hand, for all $t\in[0,\infty)\setminus Z$ we have $g(t)=\infty$ and thus $[-\tau(-\cdot)]^\ast=[g*r]=\infty$ on $[0,T)$ (not necessarily in $T$).

\emph{Step\,2.}
The previous step implies $[-\tau^z(-\cdot)]^\ast\leq[-\tau(-\cdot)]^\ast$.
Indeed, for given $t\in Z$ and $\delta>0$ let $\psi:\R\to[0,\infty)$ even and nonincreasing on $[0,\infty)$ with $\psi(0)=z^{-1}(t)$ and $E^t[\psi]\leq[-\tau(-\cdot)]^\ast(t)+\delta$.
Without loss of generality we may assume $\psi(y)<\psi(0)$ for all $y\neq0$.
Abbreviate $w=\int_\R\psi\,\d y$.
Now for any $W\geq w$ we can simply define $\tilde\psi(y)=\psi(\max\{0,|y|-\frac{W-w}{2\psi(0)}\})$, which satisfies $\int_\R\tilde\psi\,\d y=W$ and $F^c[\tilde\psi]-t\int_\R\tilde\psi\,\d y\leq E^t[\tilde\psi]=E^t[\psi]\leq[-\tau(-\cdot)]^\ast(t)+\delta$
so that we have $\tau^z(W)-tW\leq[-\tau(-\cdot)]^\ast(t)+\delta$.
On the other hand, for any $W<w$ there exists $R>0$ so that $\psi_R(y)=\psi(|y|+R)$ satisfies
\begin{equation*}
\int_\R\psi_R\,\d y=W\,,
\qquad\text{and then}\qquad
\tau^z(W)-tW
\leq F^c[\psi_R]-t\int_\R\psi_R\,\d y
=E^t[\psi_R]
\leq E^t[\psi]
\leq[-\tau(-\cdot)]^\ast(t)+\delta\,.
\end{equation*}
By the arbitrariness of $\delta>0$ we obtain $\tau^z(W)-tW\leq[-\tau(-\cdot)]^\ast(t)$ for all $W\geq0$, and taking the supremum over all $W\geq0$ the desired result follows.
Now for $t\in[0,T)$ the desired inequality is trivial due to $[-\tau(-\cdot)]^\ast(t)=\infty$, and for $t=T$ it holds by continuity.

\emph{Step\,3.}
Denote by $C\subset[0,\infty)$ the set of points $t$ such that there exists some $w\geq0$ at which $\tau^z$ is differentiable with $(\tau^z)'(w)=t$.
By \cref{thm:necessaryCondition}, for any $t\in C$ we have $[-\tau^z(-\cdot)]^\ast(t)=[g*r](t)=[-\tau(-\cdot)]^\ast(t)$.
Now consider a connected component of the complement of $C$ in $[0,\infty)$, whose interval endpoints we denote by $a$ and $b$.

If $a,b\in(0,\infty)$, then by continuity we have $[-\tau^z(-\cdot)]^\ast(t)=[-\tau(-\cdot)]^\ast(t)$ for $t=a$ and $t=b$,
and there exists some $w\geq0$ such that $a$ and $b$ are the right and left derivative, respectively, of $\tau^z$ in $w$.
It is a property of the Legendre--Fenchel conjugate that $[-\tau^z(-\cdot)]^\ast$ is linear on the interval $[a,b]$,
however, the linear interpolation between $[-\tau^z(-\cdot)]^\ast(a)$ and $[-\tau^z(-\cdot)]^\ast(b)$ is the largest possible convex function on $[a,b]$ with same function values in $a$ and $b$
so that necessarily $[-\tau(-\cdot)]^\ast\leq[-\tau^z(-\cdot)]^\ast$ on $[a,b]$.
Together with the previous step this implies $[-\tau(-\cdot)]^\ast=[-\tau^z(-\cdot)]^\ast$ on $[a,b]$.

If $a=0$, then if $a\in C$ or if also $b=0$ we can proceed exactly as before;
otherwise we have $[0,b)\subset[0,\infty)\setminus C$, which due to the admissibility of $\tau^z$ implies $[-\tau^z(-\cdot)]^\ast=\infty$ on $[0,b)$.
Therefore, the previous step again implies $[-\tau(-\cdot)]^\ast\leq[-\tau^z(-\cdot)]^\ast$ on $[a,b]$.

Finally, if $b=\infty$, then necessarily $[-\tau^z(-\cdot)]^\ast=0$ on $[a,\infty)$ and $(\tau^z)'(0)=a$.
By \cref{thm:aPrioriEstimateI} we thus have $\lim_{\phi\searrow0}z(\phi)=a$ and consequently $g=0$ on $[a,\infty)$ so that $[-\tau(-\cdot)]^\ast=g*r=0$ on $[a,\infty)$.
Summarizing, we obtain $[-\tau(-\cdot)]^\ast=[-\tau^z(-\cdot)]^\ast$ everywhere, and the desired equality $\tau^z=\tau$ immediately follows.
\end{proof}

\begin{remark}[Sufficient condition is not necessary]
By \cref{exm:urbanPlanningII}, the sufficient condition from \cref{thm:sufficientCondition} is not necessary.
\end{remark}

\paragraph{Acknowledgements.}
The author thanks Filippo Santambrogio for various discussions on the topic.
This work has been supported by by the Alfried Krupp Prize for Young University Teachers awarded by the Alfried Krupp von Bohlen und Halbach-Stiftung.

\bibliographystyle{plain}
\bibliography{urbanPlanningPhaseField}

\end{document}